\documentclass{article}
	\usepackage{textcomp,multicol,enumerate,amsmath,amssymb,latexsym,makeidx}
	\usepackage{amsthm}
	\usepackage{yhmath}
	\usepackage{amscd}
	\usepackage{amstext}
	\usepackage{amsxtra}
	\usepackage{algorithm}
	\usepackage{algpseudocode}
	\usepackage{mathastext}
	\usepackage[english]{babel} 
	\usepackage{xcolor,graphicx,subcaption}
	\makeatletter
\renewenvironment{proof}[1][\proofname]{\par
  \pushQED{\qed}%
  \normalfont 
  \topsep6\p@\@plus6\p@\relax
  \trivlist
  \item[\hskip\labelsep\bfseries #1.] 
}{%
  \popQED\endtrivlist\@endpefalse
}
\makeatother
	\usepackage{lineno}
	\usepackage[
  colorlinks,
  linkcolor=blue,
  citecolor=blue,
  urlcolor=blue,
]{hyperref}
	\usepackage{authblk}
	\usepackage{enumitem}
  \newtheorem{Th}{Theorem}[section]
  \newtheorem{lem}{Lemma}[section]
  \newtheorem{rem}{Remark}[section]
	\usepackage[msc-links]{amsrefs}
	\title{\textbf{A Source Identification Problem for the Bi-Parabolic Equation Containing a Poly-harmonic Operator}}
\author[1]{Dang Duc Trong \thanks{Email: ddtrong@hcmus.edu.vn}}
\author[2]{Bui Thanh Duy \thanks{Email: duy.buithanh@uah.edu.vn}}
\author[3]{Nguyen Dang Minh \thanks{Corresponding author, Email: minh.nd@ou.edu.vn}}
\affil[1]{Faculty of Mathematics and Computer Science, University of Science, Ho Chi Minh City, Vietnam}
\affil[1]{Vietnam National University, Ho Chi Minh City, Vietnam (VNU-HCM)}
\affil[2]{Faculty of Fundamental Science, Ho Chi Minh City University of Architecture, Ho Chi Minh City, Viet Nam}
\affil[3]{Ho Chi Minh City Open University, Ho Chi Minh City, Viet Nam}

	\allowdisplaybreaks
	
\begin{document}
	\everymath{\displaystyle}
	\maketitle
	\begin{abstract} 
	In this paper, we address the source identification problem for the bi-parabolic equation involving a  operator. Specifically, we investigate the equation $(\partial_t +\frak A)^2u(t)=\psi(t)f$, where $\frak A$ denotes a poly-harmonic operator. Given the perturbed data of $\psi$ and $u(T)$ (where $T>0$), our objective is to determine $f$. Although several scientific publications have explored regularization techniques for bi-parabolic problems, the existing literature remains limited. By relaxing certain conditions on the function $\psi$ and employing a truncation regularization method while considering the problem on an unbounded domain, we believe our results provide valuable insights.
	\end{abstract}
	\section{Introduction }
	We consider the problem of finding a source function $f$ when the function $u$ satisfying the following problem
\begin{eqnarray}
	\label{P1}
	&&u_{tt}( \mathbf{x},t)+2\frak Au_t( \mathbf{x},t)+\frak A^2u( \mathbf{x},t)= \psi(t)f(\mathbf{x}),\quad 0<t<T, \quad \mathbf{x} \in \mathbb{R}^d, \\
	\label{in-final con}
	&& u(\mathbf{x},0)=u_t(\mathbf{x},0)=0,u( \mathbf{x},T)=h(\mathbf{x}), \quad \mathbf{x} \in \mathbb{R}^d.
\end{eqnarray}
  Here, the operator $\frak A$ is defined on a dense subset $\frak D(\frak A)$ of $L^2(\mathbb{R}^d)$ and 
  $h,\psi$ are measurable data.
  In the problem, $(h, \psi)$ is in  $L^2(\mathbb{R}^d) \times L^{1}(0,T) $.
 Because the functions $h$ and $\psi$ are often unavailable and collected through measurement, the data we receive are often noisy 
 ($\widetilde h$ and $\widetilde \psi$) with the condition that  
    \begin{equation}\label{data}
    \|h-\widetilde h\|_{L^2(\mathbb{R}^d)}+\|\psi-\widetilde \psi\|_{L^{1}(0,T)}\leqslant \varepsilon,
    \end{equation}
 in which $\varepsilon>0$ is the magnitude of error. 
  
  The equation \eqref{P1} is rewritten by \[\frak L^2(u)=f\cdot\psi(t), \quad \text{where} \quad \frak Lu= (\partial_t + \frak A)u\] and the operator $\frak L$ contains a poly-harmonic operator $\frak A$, \[\frak Au:=(-\Delta)^\sigma u \quad (0<\sigma\leqslant 1).\]

	 As known, the equation $\frak Lu=0$ is well-known and the set of all $p$-th integrable solutions of the equation $\frak Lu = 0$ on the upper half space is called the $\sigma$-parabolic Bergman space (see \cite{nishio2005alpha}). The iterated parabolic operator formed $\frak L^n$ have been considered in \cite{nishio2021reproducing} and in the paper, we study the equation $\frak L^2(u)=\psi f$ because of its applicability. We call Equation \eqref{P1} by a bi-parabolic equation (see \cite{hishikawa2023remark}). The bi-parabolic equation, which models heat conduction, has numerous applications in thermal processes \cite{bulavatsky2016fractional, fushchich1990new, payne2006proposed, wang2007heat}. It is also employed to describe specific phenomena in process dynamics and filter merging \cite{bulavatskiy2003mathematical, bulavatsky2008generalized, kalantarov2009finite}. In the influential work by Fushchich, Galitsyn, and Polubinskii \cite{fushchich1990new}, the authors proposed that classical quadratic parabolic equations may not fully capture the complexities of heat and mass transfer processes, revealing certain well-known anomalies in this field. To address potential anomalies, mathematicians have developed an alternative parabolic model by replacing the second-order operator with a fourth-order operator. 

	Specifically, for $\frak A=-\Delta$, they have found that substituting the operator $\partial_t + \frak A$ with the new operator $(\partial_t + \frak A)^2$ yields the bi-parabolic equation. When $\sigma =1$, our problem involves a higher-order partial differential equation that is quadratic in both time and space variables, which has garnered significant interest from mathematicians (\cite{ebenbeck2020weak, ghoul2019construction, grunau2020positivity, pata2005strongly, segatti2020fractional, tuan2021initial, yanbing2019global}). Drawing on the work of Greer, Bertozzi, and Sapiro \cite{greer2006fourth}, we note that quadratic PDEs model a variety of natural phenomena, including ice formation, fluid dynamics in the lungs, brain warping, and the design of distinctive curves on surfaces. Along with the development of fractional calculus, in this study, we introduce a problem of determining the inverse source for the fractional Laplace operator in the above bi-parabolic equation.\par

To better understand the problem, we will apply the Fourier transform with respect to the spatial variable $\mathbf{x}$. As you can see in the next section, this will lead us to Equation          
	\begin{equation}
	\label{exact-solution}
	H_\psi(|\mathbf{z}|^{2\sigma})
    \widehat f(\mathbf{z})=\widehat h(\mathbf{z}),
	\end{equation}
    where
	\begin{equation}
	\label{H-function}
	 H_\psi(\nu)=\int_{0}^{T}e^{-(T-s)\nu}\left(T-s\right)\psi(s)ds, \quad \nu \geqslant 0. 
	\end{equation}
    We observe that the term $H_\psi(\nu) \to 0$ tends as $\nu \to \infty$. Hence, from \eqref{exact-solution}, the term causes instability, and thus the problem is ill-posed. Furthermore, although we can find $f$ by the inverse Fourier transform if  {$H_\psi(\nu) \ne 0, \forall \nu \geqslant 0$}, in many cases this term may be zero. \par 
	As is known, there have been several publications
    consider problems of finding {$f\in L^2(\mathbb{R}^d)$}
    such that
\begin{equation}\label{general}
    {\widehat f(\mathbf{z})=\frac{\widehat h(\mathbf{z})}{\mathcal{H}_{\psi}(\mathbf{z})}, \quad \text{where}\quad \mathcal{H}_{\psi}(\mathbf{z})=\int_0^T K(\mathbf{z},s)\psi(s)ds,}
\end{equation} 
  which has similar form as \eqref{exact-solution} and $K(\mathbf{z},s)\geqslant\kappa_0(\mathbf{z})>0$ (see, e.g. \cites{QN, DP} and the references there in). In this case, the mentioned papers is often concentrated on the case {$\psi(s) \ne 0 $ for all $s$}. 

	For example, in \cites{QN, DP}, the investigation was conducted under the assumption that $\psi$, the time dependent component of the source, is bounded, which guarantees a fixed sign property for $\psi$. Specifically, the authors in these papers assumed that {\[0<E_0 \leqslant |\psi(s)| \leqslant E_1, \forall s \in [0,T],\]} and consequently,  {\[E_2\leqslant |\widetilde\psi(s)| \leqslant E_3, \forall s \in [0,T].\]} 
	Under this condition, they proved that 
	{\[|\mathcal{ H}_\psi(\mathbf{z})|\geqslant \frac{1}{TE_0\kappa_0(\mathbf{z})}>0, \forall \mathbf{z} \in \mathbb{R}^d.\]} 
	Therefore, their problem is more amenable to regularization. 

	A study addressing the problem of finding the inverse source for a parabolic equation, under the condition that the function  $\psi(t) >0$, is presented in \cite{cheng2020inverse}. In this work, the author determined the source function from the measurement data $u(x,T_1)$ and  $u(x,T_2)$ with $T_1,T_2 \in (0,T)$ and $x \in \Omega_0 \subset \Omega$. Under the assumptions given in the paper, the function {$\mathcal{H}_\psi(\mathbf{z})$} is also assumed to be nonzero for all $\mathbf{z} \in \mathbb{R}^d.$

	 In the work of S. Mondal and M. T. Nair \cite{MN}, the authors also addressed a problem of determining the source for an equation containing a bi-parabolic operator. However, in their study, the function $\psi$ is fixed, and they aim to find $f$ from $\widetilde h$ such that $\|h-\widetilde h\|_{L^2(\mathbb{R}^d)}\leqslant \varepsilon$. Unlike in \cites{QN, Tuan, DP}, the source term $\psi$ in their problem is allowed to change sign on $[0,T]$. While their assumptions on $\psi$ are more general, they impose additional conditions, particularly on the smoothness of $\psi$, to prove that {$|\mathcal{H}_\psi(\mathbf{z})|>\frac{B}{|\mathbf{z}|^2}$ for all  $\mathbf{z} \ne \mathbf{0}$}. They do not consider the case where $\psi$ is perturbed by $\widetilde\psi$ because it is challenging to verify the constraints on $\widetilde\psi$. Additionally, the set of zeros of {$H_{\widetilde\psi}(\nu)$} is difficult to control. 

	Duc et al.~\cite{Duc} have also discussed this assumption in their research and obtained $\mathcal{H}_\psi(\mathbf{z})>0$ for all $\mathbf{z} \in \mathbb{R}^d$. Their study focused on the time--space 
fractional parabolic equation
	\[\partial_t^\alpha u + A^\beta u = f\cdot h(t), \qquad \alpha \in (0,1),\; \beta > 0,\]
which differs from the equation considered in \eqref{P1}. The estimates of the convergence rate in \cites{MN, Duc} obtained are of order H\"older under the assumption $H^\gamma$ ($\gamma>0$) on $f$. \par
	In summary, most existing assumptions regarding the function $\psi$ are intended to demonstrate that the magnitude of the integral in the denominator of \eqref{general} exceeds a positive constant. However, in many cases, this integral may still be zero, implying that the case {$H_\psi(\mathbf{z})=0$} can occur naturally, a situation not addressed in previous studies. Furthermore, when $\psi$ is perturbed by $\widetilde\psi$, the set of zeros of {$H_{\widetilde\psi}(\nu)$} is also not easy to evaluate, presenting a significant obstacle in the calculation process. 

	Therefore, this paper aims to regularize the problem of identifying the source $f$ in Equation \eqref{P1} subject to the conditions \eqref{in-final con}, by employing the truncation method on $\mathbb{R}^d$. The problem is reformulated in the form of \eqref{exact-solution}. We focus on examining cases where the integral term $H_\psi(\nu)$ vanishes under relaxed assumptions on the function $\psi$. Specifically, when $\|\psi-\widetilde \psi\|_{L^q(0,T)}\leqslant \varepsilon$, the set of zeros of $H_{\widetilde\psi}(\nu)$ complicates the regularization process. Under an $H^\gamma$ assumption on $f$, we provide an asymptotic estimate between the exact solution and approximate solutions in $H^s$ $(s<\gamma)$. To our knowledge, these aspects have been insufficiently investigated in previous research.

	  Now, we begin by briefly outlining the key definitions and notation used in this paper before we get to the main problem.

    For every $\mathbf{x}, \mathbf{z} \in \mathbb{R}^d$, the notation $\mathbf{x} \cdot \mathbf{z}$ denotes the Euclidean inner product in $\mathbb{R}^d$. We denote by $m$ and $m_d$ the Lebesgue measures on $\mathbb{R}$ and $\mathbb{R}^d$, respectively. Throughout the paper, $L^2(\mathbb{R}^d)$ denotes the Hilbert space of square-integrable 
functions on $\mathbb{R}^d$, equipped with the norm $\|\cdot\|_{L^2(\mathbb{R}^d)}$.

   	For any $f \in L^2(\mathbb{R}^d)$, its Fourier transform is defined by
\[
\mathcal{F}(f)(\mathbf{z}) := \widehat{f}(\mathbf{z}) = \frac{1}{\left( \sqrt{2\pi} \right)^{d}}\int_{\mathbb{R}^d} f(\mathbf{x}) e^{-i \mathbf{x} \cdot \mathbf{z}} \, d\mathbf{x}.
\]

For $s \geqslant 0$, we define
\[
H^s(\mathbb{R}^d) := \left\{ \theta \in L^2(\mathbb{R}^d) \;:\;\int_{\mathbb{R}^d} (1 + |\mathbf{z}|^2)^s \left| \widehat{\theta}(\mathbf{z}) \right|^2 \, d\mathbf{z} < \infty \right\},
\]
equipped with the norm
\[
\|\theta\|_{H^s(\mathbb{R}^d)} := \left( \int_{\mathbb{R}^d} (1 + |\mathbf{z}|^2)^s \left| \widehat{\theta}(\mathbf{z}) \right|^2 \, d\mathbf{z} \right)^{1/2}.
\]
As is well known, $H^s(\mathbb{R}^d)$ is a Hilbert space.

    \section{The ill-posed nature of the problem}
	Using the Fourier transform, we get from \eqref{P1} and \eqref{in-final con} the equation 
	\[\widehat u_{tt} (\mathbf{z},t)+2|\mathbf{z}|^{2\sigma}\widehat u_t(\mathbf{z},t)+|\mathbf{z}|^{4\sigma}\widehat u(\mathbf{z},t)=\psi (t)\widehat f({\mathbf{z}}).\]
	 Here, we note in \cites{AV, bucur} that \[\widehat{(-\Delta)^{\frac{\gamma}{2}}u}(\mathbf{z})=|\mathbf{z}|^\gamma \widehat u(\mathbf{z}).\] 
	We deduce from \eqref{in-final con} that $\widehat u(\mathbf{z},0)=\widehat u_t(\mathbf{z},0)=0$ and the solution of the above differential equation with respect to $t$ is given explicitly by 
 $$\widehat u(\mathbf{z},t)=\int_{0}^{t}\widehat f(\mathbf{z})e^{-|\mathbf{z}|^{2\sigma}(t-s)}\left(t-s\right)\psi(s)ds.$$
  At $t=T$, we have 
  $$\widehat u(\mathbf{z},T)=\int_{0}^{T}\widehat f(\mathbf{z})e^{-|\mathbf{z}|^{2\sigma}(T-s)}\left(T-s\right)\psi(s)ds.$$
This is the equation \eqref{exact-solution}. 

For all $f$ in $L^2(\mathbb{R}^d)$, if we define 
\[
\mathbf{K}f := \mathcal{F}^{-1} \big( H_\psi(|\mathbf{z}|^{2\sigma}) \, \widehat{f} \big),
\]
then equation \eqref{exact-solution} can be rewritten in the form
\[
\mathbf{K}f = h.
\]
In the sense of Hadamard, a problem is said to be \emph{ill-posed} if it fails to satisfy at least one of the following three conditions:  
	\begin{itemize}
	\item[(i)] the existence of a solution,  
	\item[(ii)] the uniqueness of that solution, and  
	\item[(iii)] the stability of the solution with respect to perturbations in the data.  
The first two conditions are clear; the third will be explained in more detail below.
	\end{itemize}
	The problem $\mathbf{K}f = h$ is said to be unstable if there exists a sequence 
$\{f_n\} \subset L^2(\mathbb{R}^d)$ such that 
\[
\mathbf{K}f_n \to \mathbf{K}f \quad \text{in } L^2(\mathbb{R}^d)
\]
but 
\[
f_n \not\to f \quad \text{in } L^2(\mathbb{R}^d).
\]
We now investigate whether our problem is ill-posed in the sense of Hadamard.

	This will be clearly demonstrated in the following theorem.   
\begin{Th}\label{ill-posed}
Let $h \in L^2(\mathbb{R}^d)$ and $\psi \in L^1(0,T)$. 

\begin{itemize}
    \item[(a)] Denote
    \[
    \mathcal{Z}_\psi := \left\{ \mathbf{z} \in \mathbb{R}^d : 
    {H_\psi(|\mathbf{z}|^{2\sigma}) = 0} \right\}.
    \]
    Equation \eqref{exact-solution} has a solution in {$L^2(\mathbb{R}^d)$} 
    if and only if 
    \[
    \widehat{h}(\mathbf{z}) = 0 \quad \text{for all } \mathbf{z} \in \mathcal{Z}_\psi
    \]
    and 
    \[
        \int_{\mathbb{R}^d \setminus \mathcal{Z}_\psi} 
        \left| \frac{\widehat{h}(\mathbf{z})}{H_\psi(|\mathbf{z}|^{2\sigma})} \right|^2
        \, d\mathbf{z} < \infty.
    \]

    \item[(b)] Denote
    \[
    N_\psi := \{ \nu \in \mathbb{R} : \psi(\nu) \neq 0 \}.
    \]
    If $m(N_\psi) > 0$, then $m_d(\mathcal{Z}_\psi) = 0$, i.e., 
    $H_\psi(|\mathbf{z}|^{2\sigma}) \neq 0$ for almost every $\mathbf{z} \in \mathbb{R}^d$,
    and equation \eqref{exact-solution} has at most one solution 
    $f \in L^2(\mathbb{R}^d)$.

    \item[(c)] For every $\beta > 0$, define
    \[
    \pi_\psi(\beta) := \left\{ \mathbf{z} \in \mathbb{R}^d :
    |H_\psi(|\mathbf{z}|^{2\sigma})| \leqslant \beta \right\}.
    \]
    If $m(N_\psi) > 0$, then {$m_d(\pi_\psi(\beta)) > 0$} for every $\beta > 0$,
    and the problem is unstable; that is, we can find a sequence 
    $\{f_n\} \subset L^2(\mathbb{R}^d)$ such that 
    \[
    \mathbf{K}f_n \to 0 \quad \text{but} \quad f_n \not\to 0.
    \]
\end{itemize}
\end{Th}

 \begin{proof}\leavevmode
\begin{enumerate}
\item[(a)] Let $f=\mathcal{F}^{-1}(F)$, where
{\[
  F(\mathbf{z}) =
  \begin{cases}
    \dfrac{\widehat h(\mathbf{z})}{H_\psi(|\mathbf{z}|^{2\sigma})}, & \mathbf{z} \notin \mathcal{Z}_\psi,\\[0.8em]
    0, & \mathbf{z} \in \mathcal{Z}_\psi.
  \end{cases}
\]}
It is straightforward to verify that $f$ is a solution of \eqref{exact-solution}.

\item[(b)] We show that $H_\psi(\nu) \not\equiv 0$ for $\nu\in\mathbb{R}$.  
Assume, to the contrary, that $H_\psi(\nu) = 0$ for every $\nu \in \mathbb{R}$. Then
\[
 { H_\psi(\nu) = \sum_{n=0}^\infty \frac{(-1)^n \nu^n}{n!}
  \int_0^T (T-s)^{n+1}\, \psi(s)\, ds,
  \quad \forall \nu \in \mathbb{R}.}
\]
It follows that
\[
 { \int_0^T t^{n+1} \psi(T-t)\, dt = 0, \quad \forall n \in \mathbb{N},}
\]
which implies
\[
  \int_0^T P(t)t \psi(T-t)\, dt = 0
\]
for every polynomial $P$. Since the set of polynomials is dense in $C[0,T]$, we deduce that $\psi = 0$ a.e., contradicting the assumption $m(N_\psi) > 0$. Hence $H_\psi \not\equiv 0$.

Moreover, since $H_\psi$ is analytic, its set of real zeros
\[
  Z_{\psi} = \{\nu \in \mathbb{R}:\, \nu \geqslant 0, \ H_\psi(\nu) = 0\}
\]
consists of isolated and at most countably many points.  
Let $\nu_j \in Z_\psi$, $j \in J \subset \mathbb{N}$, satisfy
$0 \leqslant \nu_1 < \nu_2 < \dots$. Then
\[
  \mathcal{Z}_\psi = \bigcup_{j \in J} \{\mathbf{z} \in \mathbb{R}^d:\, |\mathbf{z}|^{2\sigma} = \nu_j\}.
\]
Therefore $m_d(\mathcal{Z}_\psi) = 0$.

Now, suppose $f_1,f_2 \in L^2(\mathbb{R}^d)$ satisfy
\[
  H_\psi(|\mathbf{z}|^{2\sigma}) \, \widehat f_k(\mathbf{z})
  = \widehat h(\mathbf{z}), \quad k=1,2.
\]
This implies
\[
  H_\psi(|\mathbf{z}|^{2\sigma}) \, \widehat f_1(\mathbf{z})
  = H_\psi(|\mathbf{z}|^{2\sigma}) \, \widehat f_2(\mathbf{z}).
\]
For $\mathbf{z} \notin \mathcal{Z}_\psi$, we obtain
$\widehat f_1(\mathbf{z}) = \widehat f_2(\mathbf{z})$.  
Since $m_d(\mathcal{Z}_\psi) = 0$, it follows that
$\widehat f_1 = \widehat f_2$ a.e., and thus $f_1 = f_2$ a.e.

	\item[(c)] The proof will consist of two steps.

\textbf{Step 1. Prove that $m_d(\pi_\psi(\beta))>0$ for every $\beta>0$.}

If $m(N_\psi)>0$ then the function $H_\psi(\nu)$ is
analytic with respect to $\nu$ and $H_\psi(\nu)\not\equiv 0$. On the other hand, using the Lebesgue dominated convergence theorem yields
\[
\lim_{\nu\to\infty}H_\psi(\nu)
=\lim_{\nu\to\infty}
\int_0^T e^{-(T-s)\nu}(T-s)\psi(s)\,ds = 0.
\]
{Since $\lim_{\nu\to\infty} H_\psi(\nu) = 0$, for any $\beta > 0$ there exists $r > 0$ such that 
\[
|H_\psi(\nu)| \leqslant \beta, \quad \forall \nu \geqslant r.
\]
Therefore
\[
[r, \infty) \subset \{\nu \in \mathbb{R} : |H_\psi(\nu)| \leqslant \beta\},
\]
which implies that the set $\{\nu \in \mathbb{R} : |H_\psi(\nu)| \leqslant \beta\}$ has positive Lebesgue measure. It means
\[
m\big(\{\nu\in\mathbb{R}:~ |H_\psi(\nu)|\leqslant \beta\}\big) > 0.
\]}
{Next, we have
\[
m_d(\pi_\psi(\beta))
= \int_{\{\mathbf{z}\in\mathbb R^d:\ |H_\psi(|\mathbf z|^{2\sigma})|\leqslant\beta\}} d\mathbf z.
\]
Using spherical coordinates $\mathbf z = \rho\mathbf{w}$ with $\rho=|\mathbf z|\geqslant 0$ and $\mathbf{w}=\frac{\mathbf{z}}{|\mathbf{z}|}\in\mathbb S^{d-1}$, where 
\[
\mathbb{S}^{d-1} := \left\{ \mathbf{x} \in \mathbb{R}^d \ \big|\ |\mathbf{x}| = 1 \right\}
\]
is the unit sphere in $\mathbb{R}^d$ and $d\Sigma(\mathbf{w})$ denotes its surface measure, we have
\[
d\mathbf{z} = \rho^{d-1} \, d\rho \, d\Sigma (\mathbf{w}),
\]
so that
\[
\begin{aligned}
m_d(\pi_\psi(\beta))
&= \int_{\{\rho\geqslant0:\ |H_\psi(\rho^{2\sigma})|\leqslant\beta\}}
\left( \int_{\mathbb S^{d-1}} \rho^{\,d-1} \, d\Sigma(\mathbf{w}) \right) d\rho = \omega_d \int_{\{\rho\geqslant 0:\ |H_\psi(\rho^{2\sigma})|\leqslant\beta\}} \rho^{\,d-1}\,d\rho,
\end{aligned}
\]
where $\omega_d := \Sigma(\mathbb{S}^{d-1})=\dfrac{2\pi^{d/2}}{\Gamma(d/2)}$ is the surface area of the unit sphere in $\mathbb{R}^d$.}

{Now change variables $\nu=\rho^{2\sigma}$ (so $\rho=\nu^{1/(2\sigma)}$ and
$d\rho=\frac{1}{2\sigma}\nu^{\frac{1}{2\sigma}-1}\,d\nu$). Then
\[
\rho^{\,d-1}\,d\rho
= \nu^{\frac{d-1}{2\sigma}}\cdot\frac{1}{2\sigma}\nu^{\frac{1}{2\sigma}-1}\,d\nu
= \frac{1}{2\sigma}\nu^{\frac{d-2}{2\sigma}}\,d\nu.
\]
Hence
\[
m_d(\pi_\psi(\beta))
= \frac{\omega_d}{2\sigma}\int_{\{\nu\geqslant 0:\ |H_\psi(\nu)|\leqslant\beta\}}
\nu^{\frac{d-2}{2\sigma}}\,d\nu.
\]
Since the set $\{\nu \geqslant 0 : |H_\psi(\nu)| \leqslant \beta\}$ has positive Lebesgue measure (as argued above), and the function inside the integral, $\nu^{\frac{d-2}{2\sigma}}$, is strictly positive on a set of positive measure, it follows that $m_d(\pi_\psi(\beta))>0$ for every $\beta>0$. }

\textbf{Step 2. Prove that the problem is unstable.}

For $n,m\in\mathbb{N}$, 
we denote
\[
A_{n,m}=\left\{\mathbf{z}\in\mathbb{R}^d:~ |H_\psi(|\mathbf{z}|^{2\sigma})|\leqslant \frac{1}{n},\ |\mathbf{z}|\leqslant m\right\}.
\]
We have 
\[
\bigcup_{m=1}^\infty A_{n,m}=\pi_\psi(1/n), \quad A_{n,m} \subset A_{n,m+1}.
\]
Hence, from the fundamental properties of measure,
\[
\lim_{m\to\infty} m_d(A_{n,m}) = m_d(\pi_\psi(1/n)) > 0
\quad\text{for every } n\in\mathbb{N}.
\]
Therefore, we can find a number $m_n\in\mathbb{N}$ such that
\[
0 < m_d(A_{n,m_n}) < \infty.
\]

We choose $f = h = 0$, and set $\mathbf{K}f_n = h_n$ where
{
\[
\widehat f_n(\mathbf{z})
=\frac{\chi_{A_{n,m_n}}(\mathbf{z})}{
  \sqrt{|H_\psi(|\mathbf{z}|^{2\sigma})|\,m_d(A_{n,m_n})}},
\quad
\widehat h_n(\mathbf{z})
=\frac{\sqrt{|H_\psi(|\mathbf{z}|^{2\sigma})|}}{
\sqrt{m_d(A_{n,m_n})}}
\chi_{A_{n,m_n}}(\mathbf{z}).
\]}
We have 
\[
|\widehat f_n(\mathbf{z})| \geqslant 
\frac{\sqrt{n}}{\sqrt{m_d(A_{n,m_n})}}\,
\chi_{A_{n,m_n}}(\mathbf{z}),
\]
which implies
\[
\|\widehat f_n\|^2_{L^2(\mathbb{R}^d)} \geqslant 
\int_{\mathbb{R}^d}
\left(\frac{\sqrt{n}}{\sqrt{m_d(A_{n,m_n})}}
\chi_{A_{n,m_n}}(\mathbf{z})\right)^2
d\mathbf{z} = n.
\]
On the other hand,
\[
\|\widehat h_n\|^2_{L^2(\mathbb{R}^d)} \leqslant 
\int_{\mathbb{R}^d}
\left(\frac{1}{\sqrt{n\,m_d(A_{n,m_n})}}
\chi_{A_{n,m_n}}(\mathbf{z})\right)^2
d\mathbf{z} = \frac{1}{n}.
\]
It follows that
\[
\lim_{n\to\infty}\|h_n - h\| = 0,
\quad
\lim_{n\to\infty}\|f_n - f\| = \infty.
\]
Hence, our problem is unstable.
\end{enumerate}
 \end{proof} 
\section{Modified truncated regularization} 
Assuming that $f \in L^2(\mathbb{R}^d)$, let us consider Problem \eqref{P1}-\eqref{in-final con}. 
If it has a solution 
\[u \in C\big([0,T];L^2(\mathbb{R}^d)\big) \cap C^2\big(0,T;D({\frak A})\big),\] 
then Theorem \ref{ill-posed} implies that 
$H_\psi(|\mathbf{z}|^{2\sigma}) \not=0$ almost everywhere, and
\begin{equation}\label{1exact-solution}
    \widehat f(\mathbf{z}) = \frac{\widehat h(\mathbf{z})}{H_\psi(|\mathbf{z}|^{2\sigma})} \quad \text{a.e. in } \mathbb{R}^d.
\end{equation}

In \eqref{1exact-solution}, the term $H_\psi(|\mathbf{z}|^{2\sigma})$ is a source of numerical instability, especially when it approaches zero. There are several ways to handle this instability. 
According to regularization theory, one can choose filter functions $g_\alpha(\lambda)$ such that 
$\lim_{\alpha \to 0^+} g_\alpha(\lambda) = 1/\lambda$. Then, the term 
$1/H_\psi(|\mathbf{z}|^{2\sigma})$ can be replaced by $g_\alpha(H_\psi(|\mathbf{z}|^{2\sigma}))$. 

For example, one can take
\[
g_\alpha(\lambda) =
\begin{cases}
    \dfrac{\lambda}{\alpha+\lambda^2}, & \text{Tikhonov filter,}\\[2mm]
    \dfrac{1}{\lambda} \chi_{[\alpha,\infty)}(|\lambda|), & \text{truncated filter.}
\end{cases}
\]

Here, $\chi_\Omega$ denotes the characteristic function of $\Omega$:
\[
\chi_\Omega(\mathbf{z}) = 
\begin{cases}
1, & \mathbf{z} \in \Omega,\\
0, & \mathbf{z} \notin \Omega.
\end{cases}
\]

In this paper, we choose the truncated filter. Its main advantages are simplicity, computational efficiency, and effective noise filtering. It stabilizes the problem by eliminating small singular values that are highly susceptible to noise.

 Using the truncated method with $\lambda=H_\psi(|\mathbf{z}|^{2\sigma})$, we can use the function
 	\[ {f_{\alpha,\psi,h}=
\mathcal{F}^{-1}\left(\frac{\widehat h}{H_\psi(|\mathbf{z}|^{2\sigma})} \chi_{
    \{
    |H_\psi(|\mathbf{z}|^{2\sigma})|>\alpha\}}\right) }\] 
to approximate $f$. Since the functions $h$ and $\psi$ are unavailable, we can only use the substitute functions 
$\widetilde h$ and $\widetilde \psi$ in \eqref{data} respectively. Hence we will approximate $f$ by 
\begin{equation}\label{regularization-solution}
     {f_{\alpha,\widetilde\psi,\widetilde h}=\mathcal{F}^{-1}\left(\frac{\widehat {\widetilde h}}{H_{\widetilde\psi}(|\mathbf{z}|^{2\sigma})} \chi_{
        \{
        |H_{\widetilde\psi}(|\mathbf{z}|^{2\sigma})|>\alpha\}}\right)}.
\end{equation}
For brevity, we will denote 
\begin{equation}\label{omega}
  \omega_\psi(\alpha):=\{\mathbf{z} \in \mathbb{R}^d: |H_\psi(|\mathbf{z}|^{2\sigma})|>\alpha\}.  
 \end{equation}
The set $\omega_\psi(\alpha)$ is  used for analyzing the problem and its evaluation in the rest of the paper.
Since the functions $h$ and $\psi$ are unavailable, we can only use the substitute functions 
$\widetilde h$ and $\widetilde \psi$ in \eqref{data} respectively. Hence we will approximate $f$ by 
$R_{\alpha,\widetilde\Psi}: L^2(\mathbb{R}^d)\to L^2(\mathbb{R}^d)$ satisfying
\begin{equation}\label{regularization-solution}
  {R_{\alpha,\widetilde\Psi}(\widetilde h):=  f_{\alpha,\widetilde\psi,\widetilde h}=\mathcal{F}^{-1}\left(\frac{\widehat {\widetilde h}}{H_{\widetilde\psi}(|\mathbf{z}|^{2\sigma})} \chi_{
        \omega_{\widetilde\psi}(\alpha)}\right)}.
\end{equation}
This method could be tentatively called {\it modified truncated regularization}. Here, we present the continuity of $R_{\alpha,\widetilde\psi}\widetilde h$ with respect to $(\widetilde\psi,\widetilde h)$. Due to the nonlinear dependence of the operator $R_{\alpha,\widetilde\psi}$ on $\widetilde\psi$, establishing the continuity of this modified truncated regularization poses a significant challenge.
\begin{Th}
 Let $\alpha>0$ be fixed and $\vartheta_n,\vartheta\in L^1(\mathbb{R}^d)$, $\vartheta_n,\vartheta\not\equiv 0$ in $L^1(0,T)$.
 
 (a) The operator $R_{\alpha,\vartheta}$ is well-defined, i.e., $ R_{\alpha,\vartheta}h\in L^2(\mathbb{R}^d)$ for every $h\in L^2(\mathbb{R}^d)$.
 
 (b) If $h_n\to h$ in $L^2(\mathbb{R}^d)$ and $ 
 \vartheta_n\to\vartheta$ in $L^1(0,T)$ then
 $$   R_{\alpha,\vartheta_n} h_n=f_{\alpha,\vartheta_n,h_n}
 \stackrel{n\to\infty}{\longrightarrow}
  R_{\alpha,\vartheta}h =f_{\alpha,\vartheta,h}~~ {\rm in}~ L^2(\mathbb{R}^d).$$  
\end{Th}
\begin{proof}
The definition is well-defined since 
$$  \left|\frac{\widehat { h}}{H_{\vartheta}(|\mathbf{z}|^{2\sigma})} \chi_{
    \{
    |H_{\vartheta}(|\mathbf{z}|^{2\sigma})|>\alpha\}}\right|\leqslant \frac{1}{\alpha}|\widehat { h}|\in L^2(\mathbb{R}^d). $$ 
We consider the continuity of $R_{\alpha,\vartheta}h$ 
in  $L^2(\mathbb{R}^d)$. We have
\begin{align*}
   \|R_{\alpha,\vartheta_n}h_n-R_{\alpha,\vartheta}h\|_{L^2(\mathbb{R}^d)}^2 & \leqslant 2(J_1+J_2),
\end{align*}    
where
\begin{align*}
   J_1 &=\int_{\mathbb{R}^d} \left|(\widehat h_n-\widehat h)\frac{\chi_{\omega_{\vartheta_n}(\alpha)}}{|H_{\vartheta_n}(|\mathbf{z}|^{2\sigma})|}\right|^2 d\mathbf{z},\\
   J_2 &= \int_{\mathbb{R}^d} \left|\widehat h\left(\frac{\chi_{\omega_{\vartheta_n}(\alpha)}}{|H_{\vartheta_n}(|\mathbf{z}|^{2\sigma})|}
-\frac{\chi_{\omega_{\vartheta}(\alpha)}}{|H_{\vartheta}(|\mathbf{z}|^{2\sigma})|}
   \right)\right|^2 d\mathbf{z}.
\end{align*}
 We have
 $$  0\leqslant J_1\leqslant \frac{1}{\alpha^2}\|h_n-h\|_{L^2(\mathbb{R}^r)}^2
 \stackrel{n\to\infty}{\longrightarrow} 0.  $$
 For $J_2$, we note that
  {\[|\widehat h|^2\left|\frac{\chi_{\omega_{\vartheta_n}(\alpha)}}{|H_{\vartheta_n}(|\mathbf{z}|^{2\sigma})|}
-\frac{\chi_{\omega_{\vartheta}(\alpha)}}{|H_{\vartheta}(|\mathbf{z}|^{2\sigma})|}
\right|^2\leqslant \frac{4}{\alpha^2}|\widehat h|^2\in L^1(\mathbb{R}^d).\]}
Moreover, for $|H_\vartheta(|\mathbf{z}|^{2\sigma})|\not=\alpha$ then we have
$$  \left|\frac{\chi_{\omega_{\vartheta_n}(\alpha)}}{|H_{\vartheta_n}(|\mathbf{z}|^{2\sigma})|}
-\frac{\chi_{\omega_{\vartheta}(\alpha)}}{|H_{\vartheta}(|\mathbf{z}|^{2\sigma})|}
\right|^2\stackrel{n\to\infty}{\longrightarrow} 0. $$
Similarly to Theorem~\ref{ill-posed}, the equation $H_\vartheta(\nu)=\pm\alpha$ has at most countably many positive solutions $\{\mu_j\}_{j\in J}$, with $J\subset\mathbb{N}$. Hence
	\[{ \mathcal{Z}_{\vartheta,\alpha}:=\{\mathbf{z}\in \mathbb{R}^d: |H_\vartheta
(|\mathbf{z}|^{2\sigma})|=\alpha\}=\bigcup_{j\in J}\{|\mathbf{z}|^{2\sigma}=\mu_j\}. }\]
Hence {$m_d(\mathcal{Z}_{\vartheta,\alpha})=0$}. 
It follows that 
	\[ |\widehat h|^2\left|\frac{\chi_{\omega_{\vartheta_n}(\alpha)}}{|H_{\vartheta_n}(|\mathbf{z}|^{2\sigma})|}
-\frac{\chi_{\omega_{\vartheta}(\alpha)}}{|H_{\vartheta}(|\mathbf{z}|^{2\sigma})|}
\right|^2\stackrel{n\to\infty}{\longrightarrow} 0 \quad  \text{a.e.} \]
Using the Lebesgue dominated convergence theorem, we obtain 
$$ \lim_{n\to\infty}J_2=0. $$ 
This completes the proof of Theorem.   
\end{proof}
	\section{Convergence of truncated regularization}
In this section, our focus is on the convergence of the modified truncated regularization as $\alpha \to 0^+$, rather than on its convergence rate. Consequently, the required conditions are not overly restrictive. We now present the following theorem.
\begin{Th}
Let $\varepsilon>0$, $s \geqslant 0$, and let $\psi, \widetilde\psi \in L^1(0,T)$, 
$h, \widetilde h \in L^2(\mathbb{R}^d)$ be functions satisfying \eqref{data}, with 
$\psi \not\equiv 0$ in $L^1(0,T)$. 
Assume that the equation \eqref{exact-solution} admits a unique solution 
$f \in H^s(\mathbb{R}^d)$ as stated in Theorem~\ref{ill-posed}.  

If there exists a function $\alpha=\alpha(\varepsilon)>0$ such that  
\[
\lim_{\varepsilon\to 0^+}\alpha(\varepsilon)=0, 
\qquad 
\lim_{\varepsilon\to 0^+}\frac{\varepsilon}{[\alpha(\varepsilon)]^{1+s/{2\sigma}}}=0,
\]  
then  
\[
\lim_{\varepsilon\to 0^+}
\|f-f_{\alpha(\varepsilon),\widetilde\psi,\widetilde h}\|_{H^s(\mathbb{R}^d)}=0.
\]  
\end{Th}

\begin{rem}\leavevmode
\begin{enumerate}
\item[(i)] If we consider the problem in the $L^2(\mathbb{R}^d)$
setting (i.e., $s=0$), then we can obtain convergence under the conditions
\[
 \lim_{\varepsilon\to 0^+}\alpha(\varepsilon) = 0, 
 \quad 
 \lim_{\varepsilon\to 0^+}\frac{\varepsilon}{\alpha(\varepsilon)} = 0.  
\]

\item[(ii)] Assume that $\varepsilon = o(\alpha(\varepsilon))$. Then
\[
\lim_{\varepsilon \to 0^+} \frac{\alpha(\varepsilon)\pm C\varepsilon}{\alpha(\varepsilon)} = 1.
\]
Hence, $\alpha(\varepsilon)\pm C\varepsilon \asymp \alpha(\varepsilon)$ as $\varepsilon \to 0^+$.

\end{enumerate}
\end{rem}

\begin{proof}
Let $\alpha>0$ and $\vartheta\in L^1(0,T)$. From \eqref{omega}, we have
\[
\omega_{\vartheta}(\alpha) := \{\mathbf{z} \in \mathbb{R}^d : |H_\vartheta(|\mathbf{z}|^{2\sigma})| > \alpha \}.
\]
and clearly, we obtain $\omega_\vartheta(\alpha)=\mathbb{R}^d \setminus \pi_\vartheta(\alpha)$.  

From the regularization solution \eqref{regularization-solution} and the exact solution \eqref{1exact-solution}, we have
\[
f_{\alpha,\widetilde\psi,\widetilde h}(\mathbf{z})
= \mathcal{F}^{-1} \Bigg(
\frac{\widehat{\widetilde h}(\mathbf{z})}{H_{\widetilde\psi}(|\mathbf{z}|^{2\sigma})} 
\chi_{\{ \mathbf{z}\in \mathbb{R}^d : |H_{\widetilde\psi}(|\mathbf{z}|^{2\sigma})| > \alpha \}}
\Bigg), \quad
\widehat f(\mathbf{z}) = \frac{\widehat h(\mathbf{z})}{H_\psi(|\mathbf{z}|^{2\sigma})} .
\]
To estimate the error $\|f - f_{\alpha,\widetilde\psi,\widetilde h}\|_{H^s(\mathbb{R}^d)}$, we first split 
\[
f - f_{\alpha,\widetilde\psi,\widetilde h} 
= (f - f_{\alpha,\widetilde\psi,h}) + (f_{\alpha,\widetilde\psi,h} - f_{\alpha,\widetilde\psi,\widetilde h}),
\]
where $f_{\alpha,\widetilde\psi,h}$ is the regularized solution with the exact data $h$. 

	Applying the inequality $(a+b)^2 \leqslant 2(a^2+b^2)$ in $H^s$, we get
\[
\|f - f_{\alpha,\widetilde\psi,\widetilde h}\|_{H^s(\mathbb{R}^d)}^2\leqslant 2 \|f - f_{\alpha,\widetilde\psi,h}\|_{H^s(\mathbb{R}^d)}^2 + 2 \|f_{\alpha,\widetilde\psi,h} - f_{\alpha,\widetilde\psi,\widetilde h}\|_{H^s(\mathbb{R}^d)}^2.
\]
Next, we decompose the first term according to the set
\[
\omega_{\widetilde\psi}(\alpha) := \{\mathbf{z} \in \mathbb{R}^d : |H_{\widetilde\psi}(|\mathbf{z}|^{2\sigma})| > \alpha \}, 
\quad 
\pi_{\widetilde\psi}(\alpha) := \mathbb{R}^d \setminus \omega_{\widetilde\psi}(\alpha),
\]
so that
\begin{eqnarray*}
\|f - f_{\alpha,\widetilde\psi,h}\|_{H^s(\mathbb{R}^d)}^2 &=& \int_{\omega_{\widetilde\psi}(\alpha)} (1+|\mathbf{z}|^2)^s 
\Big|\widehat f(\mathbf{z}) - \widehat f(\mathbf{z}) \frac{H_\psi(|\mathbf{z}|^{2\sigma})}{H_{\widetilde\psi}(|\mathbf{z}|^{2\sigma})} \Big|^2 d\mathbf{z} + \int_{\pi_{\widetilde\psi}(\alpha)} (1+|\mathbf{z}|^2)^s |\widehat f(\mathbf{z})|^2 d\mathbf{z}\\ 
&=:& I_1 + I_2,
\end{eqnarray*}
and we denote
\[
I_3 := \|f_{\alpha,\widetilde\psi,h} - f_{\alpha,\widetilde\psi,\widetilde h}\|_{H^s(\mathbb{R}^d)}^2.
\]
Finally, combining the above, we obtain
\begin{equation}\label{insert}
\|f - f_{\alpha,\widetilde\psi,\widetilde h}\|_{H^s(\mathbb{R}^d)}^2 \leqslant 2 (I_1 + I_2 + I_3),
\end{equation}
where
\begin{align*}
I_1 &= \int_{\omega_{\widetilde\psi}(\alpha)} (1+|\mathbf{z}|^2)^s
\left|\widehat f(\mathbf{z}) - \widehat f(\mathbf{z}) \frac{H_\psi(|\mathbf{z}|^{2\sigma})}{H_{\widetilde\psi}(|\mathbf{z}|^{2\sigma})}\right|^2 d\mathbf{z},\\
I_2 &= \int_{\pi_{\widetilde\psi}(\alpha)} (1+|\mathbf{z}|^2)^s |\widehat f(\mathbf{z})|^2 d\mathbf{z},\\
I_3 &= \|f_{\alpha, \widetilde \psi, h} - f_{\alpha, \widetilde \psi, \widetilde h}\|_{H^s(\mathbb{R}^d)}^2.
\end{align*}
Observe that
\begin{align}\label{I12}
\|f - f_{\alpha, \widetilde \psi, h}\|_{H^s(\mathbb{R}^d)}^2 = I_1 + I_2.
\end{align}
\textbf{Step 1. Prove $\lim_{\varepsilon \to 0^+} I_1 = 0$.}

We have, for every $\psi, \vartheta \in L^1(0,T)$,
\begin{equation}\label{Lipschitz}
|H_\psi(|\mathbf{z}|^{2\sigma}) - H_\vartheta(|\mathbf{z}|^{2\sigma})| 
\leqslant T \|\psi - \vartheta\|_{L^1(0,T)}.
\end{equation}
Indeed,
\begin{align*}
|H_\psi(|\mathbf{z}|^{2\sigma}) - H_\vartheta(|\mathbf{z}|^{2\sigma})|
&= \left|\int_0^T e^{-|\mathbf{z}|^{2\sigma}(T-s)}(T-s)(\psi(s)-\vartheta(s)) ds\right|\\
&\leqslant \int_0^T (T-s) |\psi(s)-\vartheta(s)| ds \le T \|\psi - \vartheta\|_{L^1(0,T)}.
\end{align*}

Hence,
\begin{align}\label{I1-estimate}
I_1 &\leqslant \int_{\omega_{\widetilde\psi}(\alpha)} (1+|\mathbf{z}|^2)^s |\widehat f(\mathbf{z})|^2
\left|\frac{H_{\widetilde\psi}(|\mathbf{z}|^{2\sigma}) - H_\psi(|\mathbf{z}|^{2\sigma})}{H_{\widetilde\psi}(|\mathbf{z}|^{2\sigma})}\right|^2 d\mathbf{z} \nonumber\\
&\leqslant \frac{T^2 \varepsilon^2}{\alpha^2} \int_{\omega_{\widetilde\psi}(\alpha)} (1+|\mathbf{z}|^2)^s |\widehat f(\mathbf{z})|^2 d\mathbf{z} 
\leqslant \frac{T^2 \varepsilon^2}{\alpha^2} \|f\|_{H^s(\mathbb{R}^d)}^2,
\end{align}
and thus, since $\lim_{\varepsilon\to 0^+}\alpha(\varepsilon)=0,\quad \lim_{\varepsilon \to 0^+} \frac{\varepsilon}{[\alpha(\varepsilon)]^{1+s/a}}=0$, we deduce $\lim_{\varepsilon \to 0^+} I_1 = 0$.\\
\textbf{Step 2. Estimate $I_2$ and show $\lim_{\varepsilon \to 0^+} I_2 = 0$.}

By definition, for every $\mathbf{z} \in \pi_{\widetilde\psi}(\alpha)$, we have 
\[
|H_{\widetilde\psi}(|\mathbf{z}|^{2\sigma})| \leqslant \alpha.
\] 
Using \eqref{Lipschitz}, it follows that
\[
|H_\psi(|\mathbf{z}|^{2\sigma})| \leqslant |H_\psi - H_{\widetilde\psi}| + |H_{\widetilde\psi}| \leqslant \alpha + \varepsilon T.
\]
Hence, we have the inclusion
\[
\pi_{\widetilde\psi}(\alpha) \subset \pi_\psi(\alpha + \varepsilon T),
\]
which gives
\[
I_2 \leqslant \int_{\pi_\psi(\alpha + \varepsilon T)} (1+|\mathbf{z}|^2)^s |\widehat f(\mathbf{z})|^2 d\mathbf{z}=\int_{\mathbb{R}^d} \chi_{\pi_{\psi}(\alpha+\varepsilon T)}(\mathbf{z})(1+|\mathbf{z}|^2)^{s}\left|\widehat f(\mathbf{z})\right|^2d\mathbf{z}.
\]
Next, from the definition 
\[
\mathcal{Z}_\psi := \{\mathbf{z} \in \mathbb{R}^d : H_\psi(|\mathbf{z}|^{2\sigma}) = 0\},
\] 
by Theorem \ref{ill-posed}, $m(\mathcal{Z}_\psi) = 0$. We split the integral over $\pi_\psi(\alpha+\varepsilon T)$ as
\[
\pi_\psi(\alpha+\varepsilon T) =  \mathcal{Z}_\psi \,\cup\, \big(\pi_\psi(\alpha+\varepsilon T) \setminus \mathcal{Z}_\psi\big).
\]
On the set $\mathcal{Z}_\psi$, the integral is zero since $m(\mathcal{Z}_\psi)=0$. On the set $\pi_\psi(\alpha+\varepsilon T) \setminus \mathcal{Z}_\psi$, for each $\mathbf{z} \notin \mathcal{Z}_\psi$, the characteristic function satisfies
\[
\chi_{\pi_\psi(\alpha+\varepsilon T)}(\mathbf{z}) \to 0 \quad \text{as} \quad \varepsilon \to 0^+.
\]
Moreover, the function
\[
(1+|\mathbf{z}|^2)^s |\widehat f(\mathbf{z})|^2 \chi_{\pi_\psi(\alpha+\varepsilon T)}(\mathbf{z})
\]
is dominated by $(1+|\mathbf{z}|^2)^s |\widehat f(\mathbf{z})|^2 \in L^1(\mathbb{R}^d)$, which is independent of $\varepsilon$. 
Therefore, by the Lebesgue dominated convergence theorem, we conclude that
	\[\lim_{\varepsilon\to 0^+}\int_{\mathbb{R}^d} \chi_{\pi_{\psi}(\alpha+\varepsilon T)}(\mathbf{z})(1+|\mathbf{z}|^2)^{s}\left|\widehat f(\mathbf{z})\right|^2d\mathbf{z}=0.\]
Hence
\[
\lim_{\varepsilon \to 0^+} I_2 = 0.
\]
\textbf{Step 3. Prove $\lim_{\varepsilon \to 0^+} I_3 = 0$.}

We first prove that there exist $\lambda_0,z_1, a>0$  such that
\begin{equation}\label{H-upper-bound}
    |\mathbf{z}|^a|H_{\psi}(|\mathbf{z}|^{2\sigma})|\leqslant 2\lambda_0~~\text{for}~|\mathbf{z}|\geqslant z_1.  
\end{equation}
 In fact, we have
\begin{align*}
    |\mathbf{z}|^{2\sigma}|H_\psi(|\mathbf{z}|^{2\sigma})| &\leqslant \int_0^T e^{-|\mathbf{z}|^{2\sigma}(T-s)}|\mathbf{z}|^{2\sigma}(T-s)|\psi(s)|ds.
\end{align*} 
Using the inequality $ye^{-y}\leqslant e^{-1}$ for every $y \geqslant 0$ we obtain
$$ |\mathbf{z}|^{2\sigma}|H_\psi(|\mathbf{z}|^{2\sigma})|
\leqslant e^{-1}\|\psi\|_{L^1(0,T)}.    $$
So we can choose $a=2\sigma$. However, in the proof, we still use the general constant $a$ because there are many constants $a$ that satisfy \eqref{H-upper-bound}.
We have
\begin{equation}\label{second-term}
I_3 = \int_{\omega_{\widetilde\psi}(\alpha)} (1+|\mathbf{z}|^2)^s 
\left|\frac{\widehat{\widetilde h}(\mathbf{z}) - \widehat h(\mathbf{z})}{H_{\widetilde\psi}(|\mathbf{z}|^{2\sigma})}\right|^2 d\mathbf{z}.
\end{equation}
If $\mathbf{z}\in \omega_{\widetilde \psi}(\alpha)$ then 
\[  |H_{\psi}(|\mathbf{z}|^{2\sigma})|\geqslant |H_{\widetilde\psi}(|\mathbf{z}|^{2\sigma})|-|H_{\widetilde\psi}(|\mathbf{z}|^{2\sigma})-H_{\psi}(|\mathbf{z}|^{2\sigma}) |\geqslant \alpha-\varepsilon T.\]
Hence for $|\mathbf{z}|\geqslant z_1$ we have 
	\[ (\alpha-\varepsilon T)|\mathbf{z}|^a \leqslant |\mathbf{z}|^aH_{\psi}(|\mathbf{z}|^{2\sigma})| \leqslant 2\lambda_0.\]
It follows that 
	 \[ |\mathbf{z}|\leqslant \left(\frac{2\lambda_0}{\alpha-\varepsilon T}\right)^{1/a}.\]
Hence, from \eqref{second-term} and $\lim_{\varepsilon\to 0^+}\alpha(\varepsilon)=0,\quad \lim_{\varepsilon \to 0^+} \frac{\varepsilon}{[\alpha(\varepsilon)]^{1+s/a}}=0$, we deduce 
\begin{equation}\label{I3-general}
I_3 \leqslant \frac{1}{\alpha^2} \left[1 + \left(\frac{2 \lambda_0}{\alpha - \varepsilon T}\right)^{2/a}\right]^s \|\widehat{\widetilde h} - \widehat h\|_{L^2(\mathbb{R}^d)}^2 
\leqslant  \frac{C\varepsilon^2}{\alpha^{2(1+s/a)}}.
\end{equation}
Combining Steps 1--3 and \eqref{I12}, we complete the proof of the theorem.
\end{proof}

\section{Convergence rate of the regularization}

This section is devoted to a rigorous analysis of the convergence rate of the truncated regularization method. 
To obtain estimates that improve upon those presented in the previous section, a detailed study of the properties of the function $H_\psi$ is required. 
Using the definition of $H_\psi$ in \eqref{H-function}, we first establish a lemma concerning the set of its zeros.

	\begin{lem}\label{H-lemma}
  Let $P>0, \theta< 2$, $p_T\not=0$. If $\lim_{\tau\to 0^+}\tau^\theta\psi(T-\tau)= p_T,
     |\tau^\theta\psi(T-\tau)|\leqslant P$ then we have
$$ \lim_{x\to\infty}x^{-\theta+2}H_\psi(x)=p_T\Gamma(-\theta+2),  $$
where $\Gamma(z)=\int_0^\infty t^{z-1}e^{-t}dt$ is the Gamma function. 
\end{lem}
\begin{rem}
	\begin{enumerate}\leavevmode
    \item[(i)] We note that the set of $\psi$ is not small, for example: $\psi(t)=(T-t)^{-\theta}[b_0+(T-t)\xi(t)]$ where $b_0 \ne 0,\theta < 1$, and $\xi \in L^\infty(0,T)$.
    \item[(ii)] From Lemma \ref{H-lemma}, we can put $a=2\sigma(2-\theta)$ to obtain \eqref{H-upper-bound}. 
	\end{enumerate}
\end{rem}
\begin{proof} 
	By putting $y=x(T-s)$, we obtain 
     $$  H(x)=\int_0^{xT} e^{-y}\frac{y}{x}\psi\left(T-\frac{y}{x}\right)
     \frac{1}{x}dy=
     \int_0^{xT} e^{-y}\left(\frac{y}{x}\right)^{-\theta+1}
     \left(\frac{y}{x}\right)^{\theta}\psi
     \left(T-\frac{y}{x}\right)
     \frac{1}{x}dy. $$
     Hence 
     {\[  x^{2-\theta}H(x)=\int_0^{xT} e^{-y}y^{-\theta+1}
     \left(\frac{y}{x}\right)^{\theta}\psi\left(T-\frac{y}{x}\right)
     dy. \]}
     Using the Lebesgue dominated convergence theorem gives
       {\[ \lim_{x\to\infty}x^{-\theta+2}H(x)=\int_0^\infty 
     e^{-y}y^{(-\theta+2)-1}p_Tdy=p_T\Gamma(-\theta+2),\]}
     where $\Gamma(z)=\int_0^\infty t^{z-1}e^{-t}dt$ is the Gamma function.

\end{proof}   

\begin{lem}\label{l1}
Let $h_0>0$ be a constant such that
\begin{equation*} 
    \frac{1}{2}\,\nu^{\theta-2}\,|p_T \Gamma(2-\theta)|
    \leqslant |H_{\psi}(\nu)| \leqslant 2\,\nu^{\theta-2}\,|p_T \Gamma(2-\theta)| 
    \quad \text{for all } \nu \geqslant h_0.
\end{equation*} 

Define
\begin{align*}
    B_\psi(\rho) &= \{\mathbf{z} \in \mathbb{R}^d : |H_\psi(|\mathbf{z}|^{2\sigma})| \leqslant \rho,~ |\mathbf{z}|^{2\sigma} \leqslant h_0\},\\
    C_\psi(\rho) &= \{\mathbf{z} \in \mathbb{R}^d : |H_\psi(|\mathbf{z}|^{2\sigma})| \leqslant \rho,~ |\mathbf{z}|^{2\sigma} > h_0\}.
\end{align*}

Then, for $\pi_\psi(\rho)$ defined in Theorem \ref{ill-posed}, we have
\[
\pi_\psi(\rho) = B_\psi(\rho) \cup C_\psi(\rho), \quad B_\psi(\rho) \cap C_\psi(\rho) = \emptyset,
\]
and
	\begin{enumerate}
	\item[(a)] for 
\[
C_1 := \left( \frac{1}{2} |p_T| \Gamma(2-\theta) \right)^{\frac{1}{2\sigma(2-\theta)}},
\]
we have
\[
C_\psi(\rho) \subset \{\mathbf{z} \in \mathbb{R}^d : |\mathbf{z}| > C_1 \rho^{-\frac{1}{2\sigma(2-\theta)}} \},
\]

	\item[(b)] there exist constants $\beta, \rho_0, C_0 > 0$, independent of $\rho$, such that for all $0 < \rho < \rho_0$, 
\[
m_d[B_\psi(\rho)] \leqslant C_0 \rho^\beta,
\]
where $m_d$ denotes the Lebesgue measure in $\mathbb{R}^d$.
	\end{enumerate}
\end{lem}

 \begin{proof}
From the definition of $\pi_\psi(\rho), B_\psi(\rho), C_\psi(\rho)$ we obtain $\pi_\psi(\rho)=B_\psi(\rho)\cup C_\psi(\rho)$ and $B_\psi(\rho)\cap C_\psi(\rho)=\emptyset$.\\
\textbf{We prove Part (a).} 

Letting $\mathbf{z}\in C_\psi(\rho)$, we obtain $|\mathbf{z}|^{2\sigma}>h_0$ and 
$$ C_1^{2\sigma(2-\theta)}|\mathbf{z}|^{2\sigma(\theta-2)}= \frac{1}{2}|\mathbf{z}|^{2\sigma(\theta-2)}|p_T| \Gamma(2-\theta)\leqslant H_\psi(|\mathbf{z}|^{2\sigma})<\rho. $$ 
it follows that 
$$ |\mathbf{z}|\geqslant C_1\rho^{-\frac{1}{2\sigma(2-\theta)}}.  $$
This completes the proof of Part (a).\\
\textbf{We consider Part (b).}  

Using Lemma \ref{H-function} gives $|H_\psi(\nu)|\not=0$ as $|\nu|\geqslant h_0$, $\nu\in \mathbb{R} $. So, by the analyticity of $H_\psi(z), z\in\mathbb{C}$, it has only finite positive zeros $\nu_k, k=\overline{1,N}$ on the real axis and we can describe $H_\psi(\nu)$ in the following form 
  $$H_\psi(\nu)=\widetilde H_\psi(\nu) \prod_{k=1}^N\left(\nu-\nu_k\right)^{p_k}~~ 
 \text{ where}~ \widetilde H_\psi(\nu) \ne 0, \quad \text{and} \quad 0\leqslant\nu \leqslant h_0.$$ 
  Since  $\widetilde H_\psi(\nu) \neq 0$ for $0\leqslant \nu \leqslant h_0$, there is $M_0>0$ such that 
  $$|\widetilde H_\psi(\nu)| \geqslant M_0~ \text{for all}~ 0\leqslant\nu \leqslant h_0.$$ 
  If $\nu_k<0$ then $|\nu -\nu_k|\geqslant |\nu_{k}|$. Therefore, without loss of generality, we can assume that 
  $$0 \leqslant \nu_{1}<\nu_{2}<\cdots<\nu_N<h_0$$ 
  and thus 
  $$|H_\psi(\nu)|\geqslant M_0 \prod_{k=1}^{N}\left|\nu-\nu_k\right|^{p_k} 
  ~\text{where~} p_k\in\mathbb{N}, k=1,\ldots,N.$$
   Put 
   $$L=\min _{1 \leqslant k \leqslant N-1}\left(\nu_{k+1}-\nu_{k}\right), P_N=\sum_{k=1}^{N} p_{k}.$$
    For $1 \leqslant k \leqslant N$, let $\delta_k=\frac{\rho^\frac{1}{2p_k}} {M_0^{\frac{1}{2p_k}} L^{\frac{P_N-2 p_k}{2 p_k}}}$, we have threre cases as follows:\\
  \textbf{Case 1.} For $k=1,...,N-1$, if $\nu_k+\delta_{k} \leqslant \nu \leqslant \nu_{k+1}-\delta_{k+1}$, we have $$|H_\psi(\nu)| \geqslant M_0 \prod_{k=1}^{N}\left|\nu-\nu_{k}\right|^{p_k} \geqslant M_0 \delta_{k}^{p_k} \delta_{k+1}^{p_{k+1}} L^{\mathbb{P}_k}=\rho.$$
   Here, $\mathbb{P}_{k}=P_N-p_k-p_{k+1}$.\\
  \textbf{Case 2.} $\nu_N+\delta_N<\nu$, we have $$|H_\psi(\nu)|  \geqslant M_0 L^{P_N-p_N} \delta_N^{p_N}$$
   and by the choice $\delta_N=\left(\frac{\rho L^{2p_N}}{L^{P_N} M_0}\right)^{\frac{1}{2 p_N}}$ with $\rho \leqslant M_0L^{P_N}$, we get $|H_\psi(\nu)|  \geqslant \rho$.\\
  \textbf{Case 3.} If $0 \leqslant \nu<\nu_{1}-\delta_{1}$, similarly, 
  $$|H_\psi(\nu)|  \geqslant M_0 L^{P_N-p_1} \delta_1^{p_1}$$ and for $\delta_{1}=\left(\frac{\rho L^{2p_N}}{L^{P_N} M_0}\right)^{\frac{1}{2 p_1}}$, we obtain that $|H_\psi(\nu)|  \geqslant \rho$. 
  
  Therefore, $$B_\psi(\rho) \subset \bigcup_{k=1}^{N}V_k~~
   \text{where} ~V_k=(\nu_k-\delta_k,\nu_k+\delta_k).$$ Next, as shown in \cite{jorgensen2014}, the volume of a ball of radius $\rho$ is given by $V_{d}(\rho)=\frac{\pi^{\frac{d}{2}}}{\Gamma\!\left(\tfrac{d}{2}+1\right)}\,\rho^d,$ where $\Gamma$ denotes the Gamma function. So, we have 
  \begin{eqnarray*}
      m_d[B_\psi(R)]&\leqslant&\sum_{k=1}^N \frac{\pi^{\frac{d}{2}}}{\Gamma\left(\frac{d}{2}+1\right)}[(\nu_k+\delta_k)^d-(\nu_k-\delta_k)^d]=\frac{2\pi^{\frac{d}{2}}}{\Gamma\left(\frac{d}{2}+1\right)}\sum_{k = 1}^N \delta _k \sum_{j=0}^{d-1}(\nu_k+\delta_k)^{d-1-j}(\nu_k-\delta_k)^j\\
      &\leqslant&\frac{2d\pi^{\frac{d}{2}}}{\Gamma\left(\frac{d}{2}+1\right)}\sum_{k = 1}^N \delta _k (\nu_k+\delta_k)^{d-1}\leqslant\frac{2d\pi^{\frac{d}{2}}}{\Gamma\left(\frac{d}{2}+1\right)}\sum_{k = 1}^N \delta _k \left(\max_{k=\overline{1,N}} \nu_k+\delta_k\right)^{d-1}\\
      &\leqslant&\frac{2d\pi^{\frac{d}{2}}}{\Gamma\left(\frac{d}{2}+1\right)}\sum_{k = 1}^N \delta _k (\nu_k+\delta_k)^{d-1}\leqslant\frac{2d\pi^{\frac{d}{2}}}{\Gamma\left(\frac{d}{2}+1\right)}\sum_{k = 1}^N \sum_{j=0}^{d-1}\left( \begin{array}{*{20}{c}}
          {d-1} \\ 
          {j} 
      \end{array}\right) \widetilde{\nu}^{d-1-j} \delta_k^{j+1}.
  \end{eqnarray*}
  Here, $\widetilde \nu =\max_{k=\overline{1,N}} \nu_k$. Choose 
  $$\rho_0=\min\left\{ M_0 L^{P_N}, \frac{1}{2}\right\}, \beta=\min _{k=\overline{1,N}, j=\overline{0,d-1}}\frac{j+1}{2 m_{s}},$$
   then for $\rho \in (0,\rho_0)$, we obtain $\mu[B_\psi(\rho)]\leqslant C_0\rho^{\beta}$ where $$C_0= \frac{2d\pi^{\frac{d}{2}}}{\Gamma\left(\frac{d}{2}+1\right)}\sum_{k = 1}^N \sum_{j=0}^{d-1}\left( \begin{array}{*{20}{c}}
      {d-1} \\ 
      {j} 
  \end{array}\right) \frac{\widetilde{\nu}^{d-1-j}} { \left(M_0^{\frac{1}{2p_k}} L^{\frac{P_N-2 p_k}{2 p_k}}\right)^{j+1}}.$$
   and the proof of this lemma is finished.
  
 \end{proof}   
    
	Using the lemmas above, we obtain the result on the convergence rate of truncated regularization. For the convenience of the readers, we explicitly state the assumptions that were made in the previous results.
	
	\begin{Th}
Let $d\geqslant 1$, $\sigma>0$, $\varepsilon>0$, $s\geqslant 0$, and let 
$\psi\in L^1(0,T)$, $h\in L^2(\mathbb{R}^d)$ with $\psi\not\equiv 0$ in $L^1(0,T)$.
Let $\beta>0$ be as in Lemma~\ref{l1}\,(c). Assume that:
\begin{itemize}
  \item[(a)] There exist $\widetilde\psi\in L^1(0,T)$ and 
  $\widetilde h\in L^2(\mathbb{R}^d)$ depending on $\varepsilon$ such that
  \[
  \|\psi-\widetilde\psi\|_{L^1(0,T)} + \|h-\widetilde h\|_{L^2(\mathbb{R}^d)} \leqslant \varepsilon
  \quad \text{for all sufficiently small } \varepsilon>0.
  \]
  \item[(b)] There exist $p\in(1,2)$ and $\gamma>0$ such that
  $f \in L^p(\mathbb{R}^d)\cap H^{\gamma}(\mathbb{R}^d)$; denote by $q=\frac{p}{p-1}$ the conjugate exponent of $p$.
  \item[(c)] There exist $P>0$, $\theta<2$, and $p_T\neq 0$ such that
  \[
  \lim_{\tau\to 0^+}\tau^\theta \psi(T-\tau)=p_T, 
  \qquad 
  \sup_{0<\tau<T}\big|\tau^\theta \psi(T-\tau)\big|\leqslant P.
  \]
\end{itemize}
Then, for any $s\in[0,\gamma]$, there exists a constant $C>0$ independent of $\varepsilon$ such that
\[
\|f - f_{\alpha(\varepsilon),\widetilde\psi,\widetilde h}\|_{H^s(\mathbb{R}^d)}^2
\leqslant C\, \varepsilon^{\frac{2b}{\,1+\frac{s}{2-\theta}+b\,}},
\]
where
\[ \alpha:=\alpha(\varepsilon)=\varepsilon^{\frac{1}{1+\frac{s}{2-\theta}+b}}\]
and
\[
b=\min\left\{ \frac{\beta(q-2)}{2q}, \ \frac{\gamma-s}{2\sigma(2-\theta)}\right\}.
\]
\end{Th}

	\begin{proof}
    From \eqref{insert}, we have 
$$  \|f-f_{\varepsilon,\widetilde\psi,\widetilde h}\|_{H^s(\mathbb{R}^d)}^2=2(I_1+I_2+I_3).   $$
As shown in \eqref{I1-estimate}, we have
\begin{equation*}
    I_1\leqslant \frac{T^2\varepsilon^2}{\alpha^2}\|f\|_{H^s(\mathbb{R}^d)}^2.
\end{equation*}    
We estimate $I_2$. From Lemma \ref{l1}, we have
{ \[I_2=I_{in}+I_{out}\]}
where
\begin{align*}
  I_{in}&=\int_{B_{\widetilde\psi}(\alpha)}
  (1+|\mathbf{z}|^2)^{s}\left|\widehat f(\mathbf{z})\right|^2d\mathbf{z},\\
  I_{out}&=\int_{C_{\widetilde\psi}(\alpha)}
  (1+|\mathbf{z}|^2)^{s}\left|\widehat f(\mathbf{z})\right|^2d\mathbf{z}.
\end{align*}
Since $f\in L^p(\mathbb{R}^d)$ for $1<p< 2$, we obtain in view of Hausdorff-Young's theorem that $\widehat f\in L^q(\mathbb{R}^d)$ with $q> 2$, $\frac{1}{p}+\frac{1}{q}=1$ and $\|\widehat f\|_{L^q(\mathbb{R}^d)}\leqslant \|f\|_{L^p(\mathbb{R}^d)}$.
Using this inequality together with H\"older's inequality, we obtain
{\begin{align*}
   I_{in} &\leqslant (1+h_0^{1/\sigma})^s(m_d(B_{\widetilde \psi}(\alpha)))^{(q-2)/q}\left(\int_{B_{\widetilde \psi}(\alpha)}|\widehat f(\mathbf{z})|^q
   d\mathbf{z}\right)^{2/q}\\
   &\leqslant (1+h_0^{1/\sigma})^s(m_d(B_{\widetilde \psi}(\alpha)))^{(q-2)/q}\|f\|_{L^p(\mathbb{R}^d)}^2.
\end{align*}}  
{Next, in order to handle the measure of the set $B_{\widetilde\psi}(\alpha)$, we note that}
\begin{equation}\label{B-set}
 B_{\widetilde\psi}(\alpha)\subset B_{\psi}(\alpha+\varepsilon T).   
\end{equation}
 From Lemma \ref{l1} we obtain
 \[m_d(B_{\widetilde\psi}(\alpha))\leqslant m_d(B_{\psi}(\alpha+\varepsilon T))\leqslant C(\alpha+\varepsilon T)^\beta.\]
It follows that
\begin{equation}\label{Iin-estimate}
   I_{in}\leqslant C(\alpha+\varepsilon T)^{\frac{\beta(q-2)}{q}}\|f\|_{L^p(\mathbb{R}^d)}^2. 
\end{equation}
To compute $I_{out}$, we consider the set $C_{\widetilde\psi}(\alpha)$. Let $\mathbf{z}\in C_{\widetilde\psi}(\alpha)$ and $|\mathbf{z}|>h_0$. Using Lemma \ref{l1} and proceeding as in \eqref{B-set}, we obtain
$$ C_{\widetilde\psi}(\alpha)\subset C_\psi(\alpha+\varepsilon T)
\subset \{\mathbf{z}\in\mathbb{R}^d: |\mathbf{z}|> C_{1} (\alpha+\varepsilon T)^{-\frac{1}{2\sigma(2-\theta)}}\}.  $$
Hence, 
\begin{align}
   I_{out} &\leqslant \int_{\{|\mathbf{z}|> C_{1} (\alpha+\varepsilon T)^{-\frac{1}{2\sigma(2-\theta)}} \}} (1+|\mathbf{z}|^2)^{s-\gamma}
   (1+|\mathbf{z}|^2)^{\gamma}|\widehat f(\mathbf{z})|^2
   d\mathbf{z}\nonumber\\
   &\leqslant C(\alpha+\varepsilon T)^{\frac{\gamma-s}{\sigma(2-\theta)}} \|f\|_{H^\gamma(\mathbb{R}^d)}^2  \label{Iout-estimate}.
\end{align}
Finally, we estimate $I_3$. Using \eqref{I3-general} with 
$a=2-\theta$ yields
\begin{align}
    \|f_{\alpha, \widetilde \psi, h}-f_{\alpha, \widetilde \psi,\widetilde h}\|_{H^s(\mathbb{R}^d)}^2&\leqslant C\frac{\varepsilon^2}{\alpha^{2(1+s/(2-\theta))}}.
    \label{I3-estimate}
\end{align} 
 Combining \eqref{I1-estimate}, \eqref{Iin-estimate},
 \eqref{Iout-estimate} and \eqref{I3-estimate} we obtain
 \begin{align*}
    \|f-f_{\varepsilon,\widetilde\psi,\widetilde h}\|_{H^s(\mathbb{R}^d)}^2
    &\leqslant 2\frac{T^2\varepsilon^2}{\alpha^2}\|f\|_{H^s(\mathbb{R}^d)}^2+
    C(\alpha+\varepsilon T)^{\frac{\beta(q-2)}{q}}\|f\|_{L^p(\mathbb{R}^d)}^2+
    C(\alpha+\varepsilon T)^{\frac{\gamma-s}{\sigma(2-\theta)}} \|f\|_{H^\gamma(\mathbb{R}^d)}^2 
    +C\frac{\varepsilon^2}{\alpha^{2(1+s/(2-\theta))}}\\
    & \leqslant C\frac{\varepsilon^2}{\alpha^{2(1+s/(2-\theta))}}\left(1+\|f\|_{H^\gamma(\mathbb{R}^d)}^2\right)+C(\alpha+\varepsilon T)^{2b}
    \left(\|f\|_{H^\gamma(\mathbb{R}^d)}^2+\|f\|_{L^p(\mathbb{R}^d)}^2\right)  
 \end{align*}
 where $b=\min\left\{ \frac{\beta(q-2)}{2q}, \frac{\gamma-s}{2\sigma(2-\theta)}\right\}$. \\
We choose $\alpha$ by balancing 
 	\[\frac{\varepsilon^2}{\alpha^{2(1+s/(2-\theta))}}
 =\alpha^{2b}\]
 which gives
  \[ \alpha=\varepsilon^{\frac{1}{1+\frac{s}{2-\theta}+b}}\]
 and 
  \[\|f-f_{\varepsilon,\widetilde\psi,\widetilde h}\|_{H^s(\mathbb{R}^d)}^2\leqslant C\varepsilon^{\frac{2b}{1+\frac{s}{2-\theta}+b}}.\]
 \end{proof}
 \section{A numerical example}
 Given that the theoretical results we have presented are grounded in the Fourier transform, we will therefore employ a discrete Fourier transform algorithm analogous to the one in paper \cite{doi:10.1137/0915067}. For $T=\sigma=1$ and $\psi(t)=1$, we consider
$$ u_{tt}(\mathbf{x}, t) - 2\Delta u_t(\mathbf{x}, t) + \Delta^2 u(\mathbf{x}, t) = \psi(t)f(\mathbf{x}), \quad 0 < t < 1, \quad \mathbf{x} \in \mathbb{R},$$
with the initial condition $u(x,0)=u_t(x,0)=0$ and the final condition
$$h(x) = \int_0^1 \dfrac{s}{\sqrt{4s+1}}e^{-\frac{x^2}{4s+1}}ds.$$
Let $A,B\in\mathbb{R}$ with $A<B$ and $h(x)\approx 0$ for $x\notin[A,B]$. Define the spatial domain and parameters by considering the interval $[A,B]$, partitioned into a vector of $L$ evenly spaced points $\mathbf{x}=[x_j]_{j=1}^L$ where $\mathbf{x}$ is a vector of length $L$, with each element $x_j=A+(j-1)\frac{B-A}{L-1}$. The noisy data is generated as follows
$$h_\epsilon(x_j) = h(x_j) + \epsilon\cdot Y_j,$$
where $Y_j\sim\mathcal{U}(-1;1)$ and $\epsilon$ is the strength of the noise. Figure \ref{fig:data} illustrates the noisy data generated from $h(x)$. The continuous curve in the plot shows the true data. The discrete points represent the corresponding data observed at the coordinates $x_j$, which include noise.
\begin{figure}[H]
    \centering
    \includegraphics[width=0.7\textwidth]{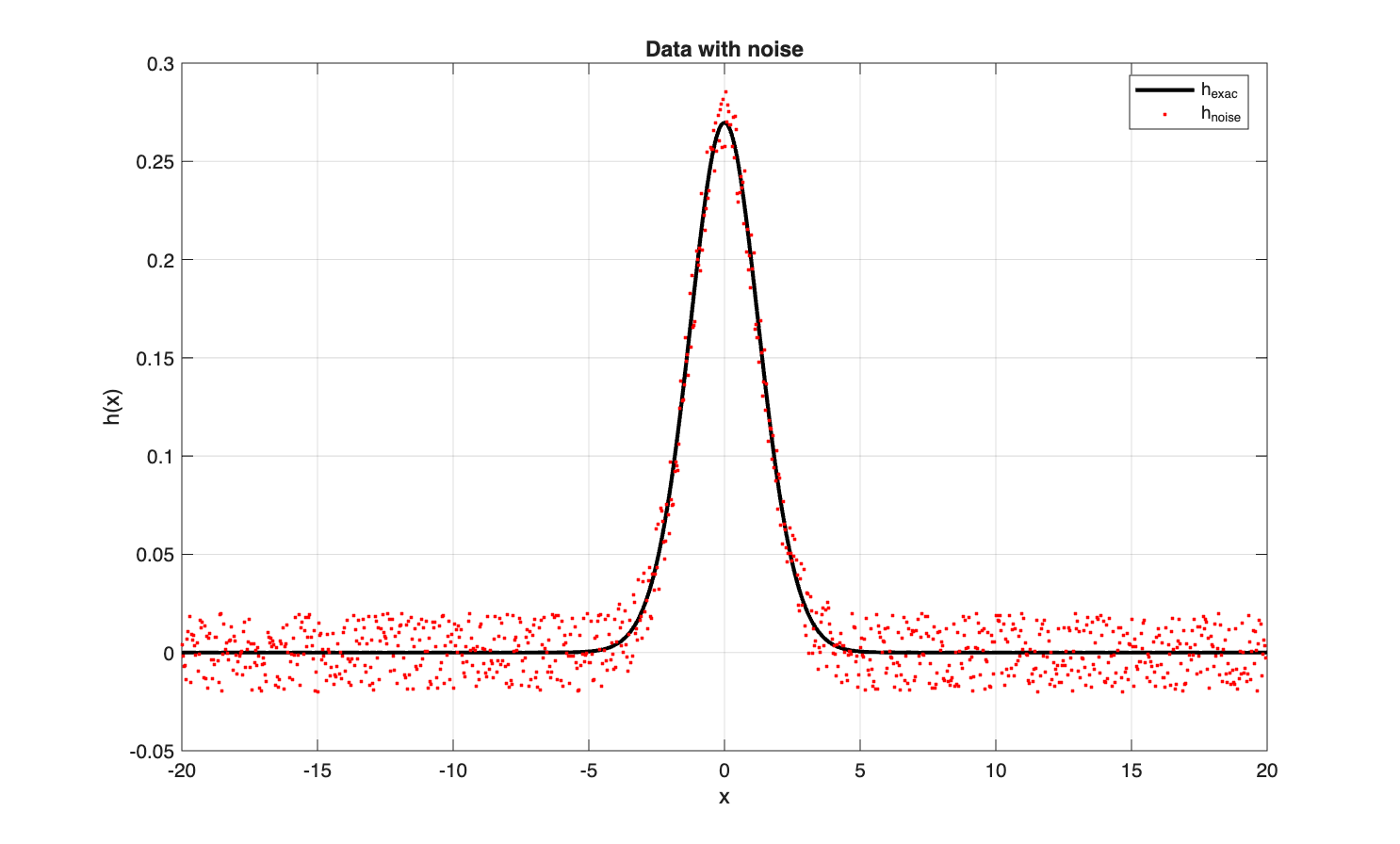}
    \caption{The noisy data generated from $h(x)$ with $\epsilon=10^{-2}$.}
    \label{fig:data}
\end{figure}
A comparison between the regularized and the exact solutions is presented in Figure \ref{fig:sol} under several noise conditions. The subplots, from left to right, correspond to error magnitudes of $\epsilon=10^{-1}$ to $10^{-3}$, respectively. In each plot, the exact solution is depicted by the bold continuous curve, whereas the regularized solution is represented by the dashed line.
\begin{figure}[htb!] 
    \centering 
    \begin{subfigure}[b]{0.5\textwidth}
        \centering
        \includegraphics[width=\linewidth]{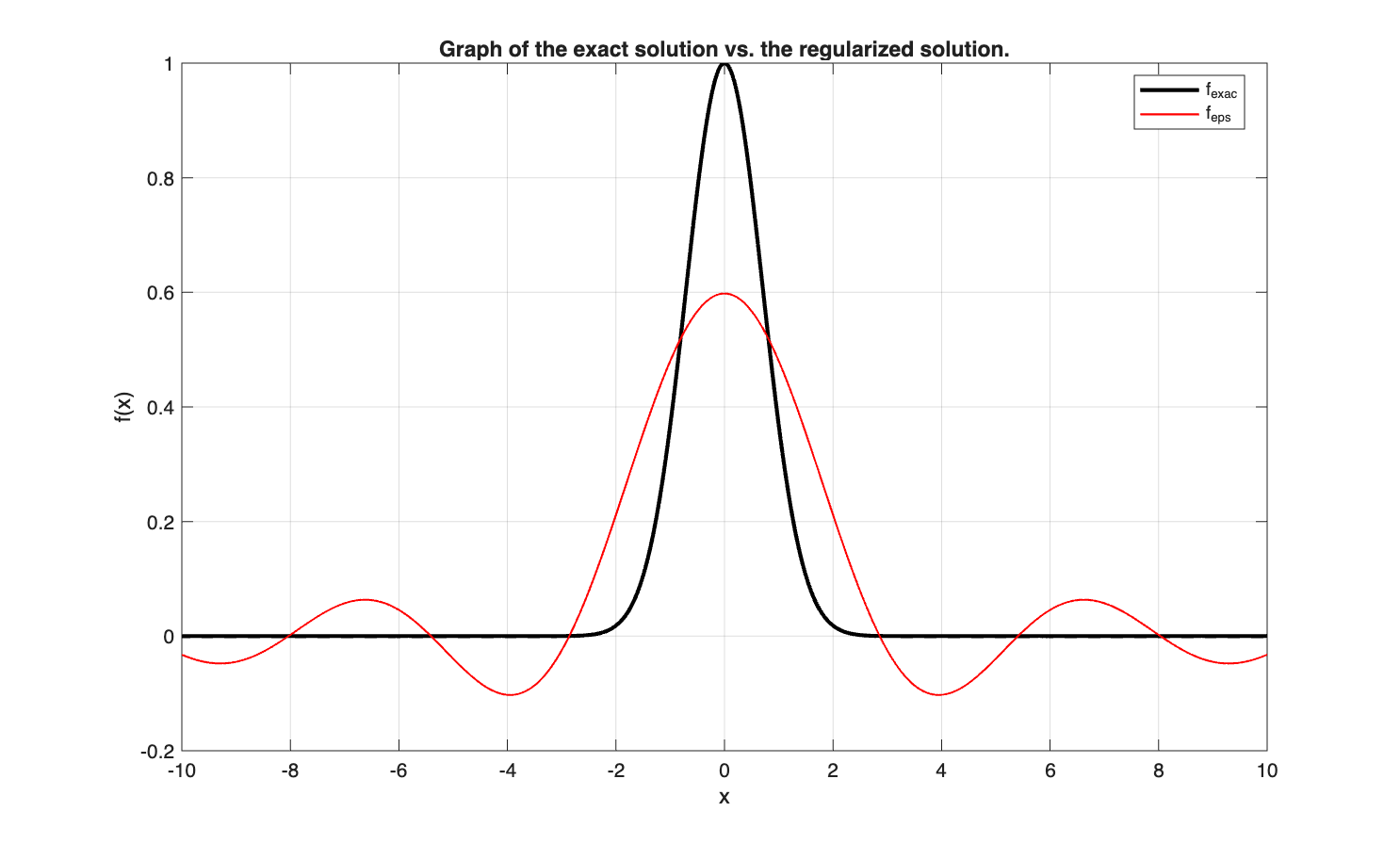}
        \caption{$\epsilon=10^{-1}$}
        \label{fig:sol1e-1}
    \end{subfigure}
    \begin{subfigure}[b]{0.5\textwidth}
        \centering
        \includegraphics[width=\linewidth]{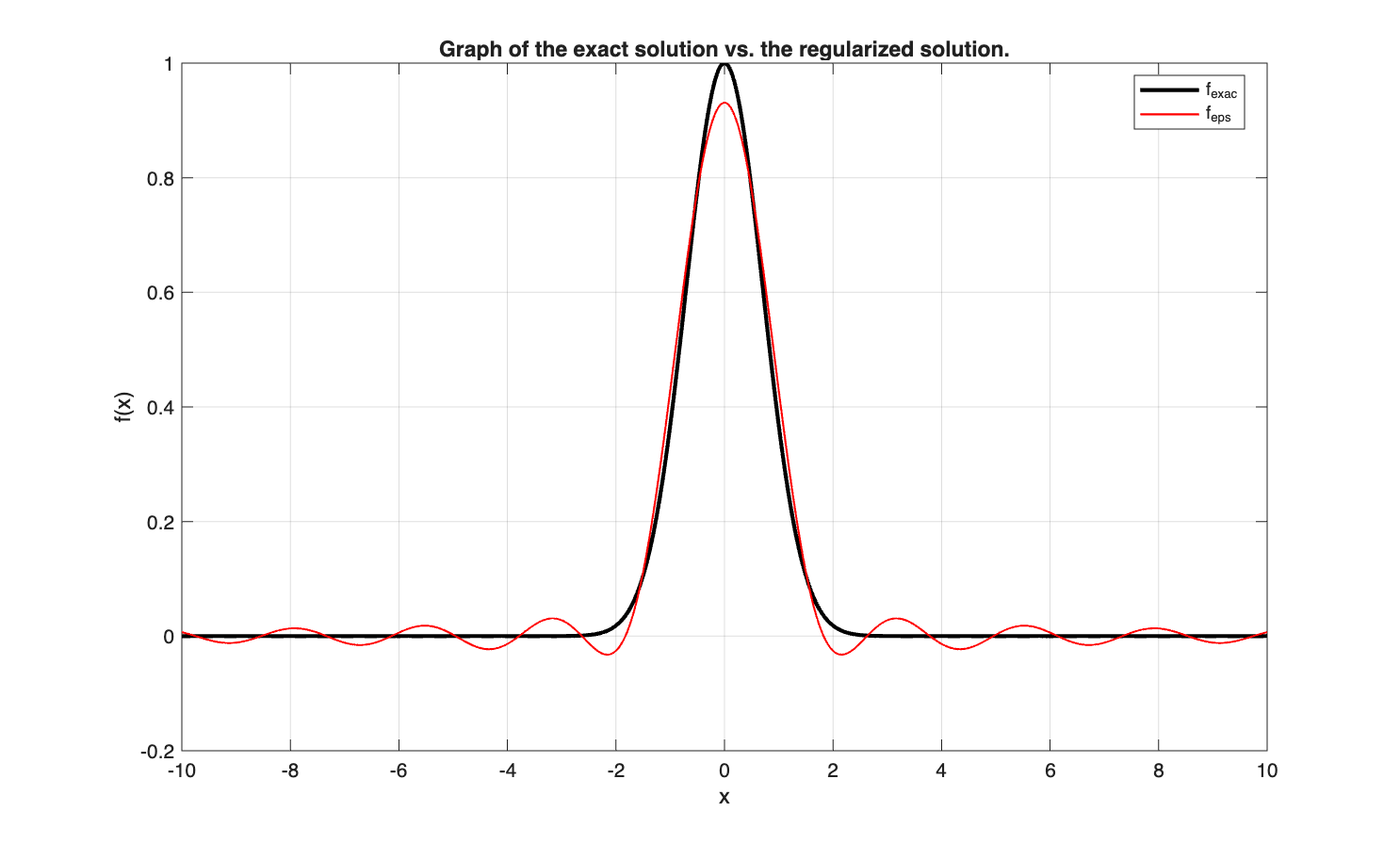}
        \caption{$\epsilon=10^{-2}$}
        \label{fig:sol1e-2}
    \end{subfigure}
    \begin{subfigure}[b]{0.5\textwidth}
        \centering
        \includegraphics[width=\linewidth]{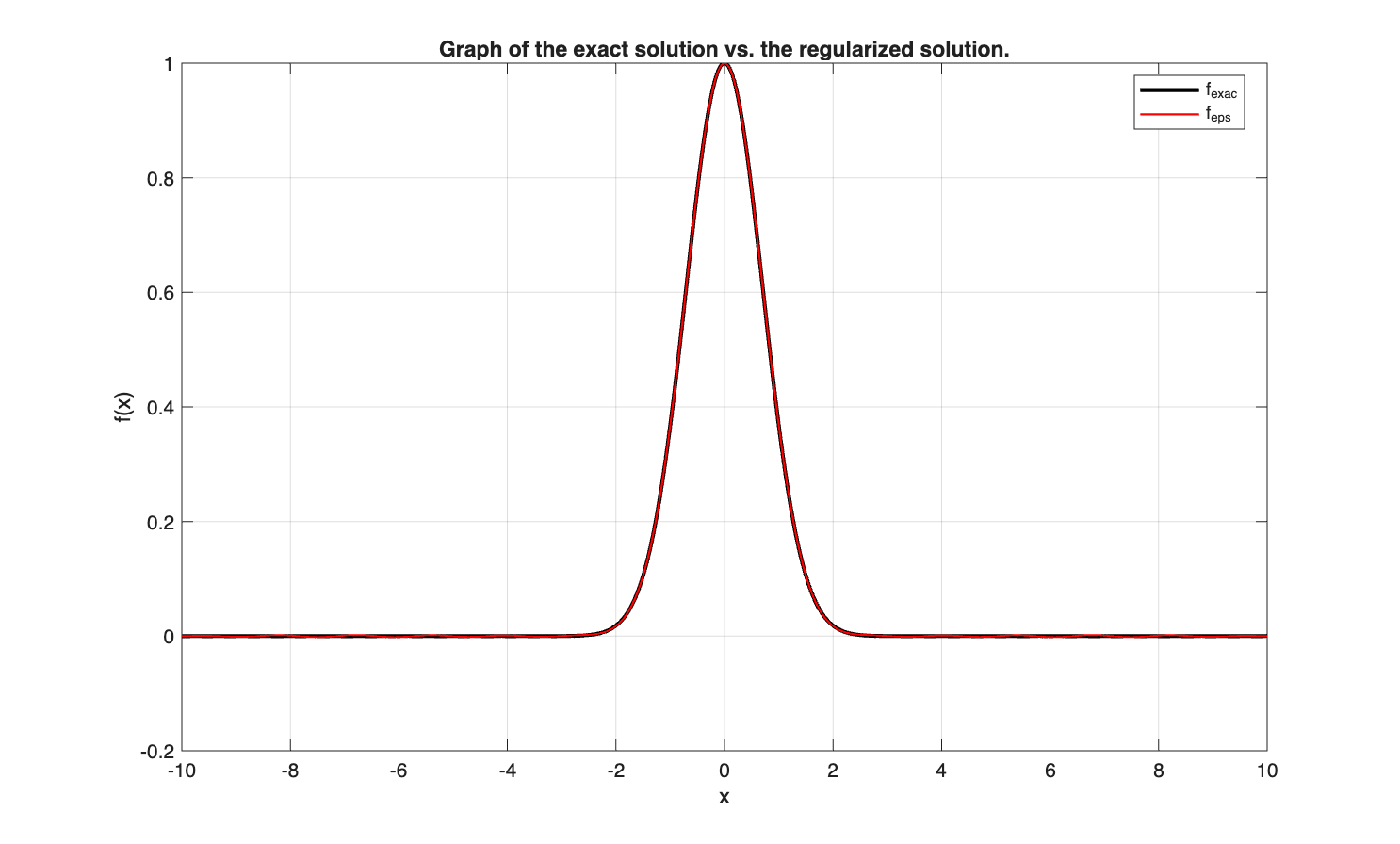}
        \caption{$\epsilon=10^{-3}$}
        \label{fig:sol1e-3}
    \end{subfigure}
    \caption{Graph of the exact solution vs. the regularized solution.}
    \label{fig:sol}
\end{figure}

Furthermore, to compute the $L^2(\mathbb{R})$ error, we use the following discrete approximation
$$
E_\text{abs} = \left\| f_\text{exact} - f_\epsilon\right\|_{L^2} \approx \sqrt{ dx \cdot \sum_{i=1}^{L} \left( f_\text{exact}((x_j)) - f_\epsilon(x_j) \right)^2 },
$$
where $dx=\frac{B-A}{L-1}$.  Moreover, to contextualize the approximation error with respect to the magnitude of the exact solution, we employ the notion of relative error, defined as
\[
E_{\text{rel}} = \frac{\left\| f_{\text{exact}} - f_\epsilon \right\|_{L^2}}{\left\| f_{\text{exact}} \right\|_{L^2}}
\approx
\frac{
    \sqrt{\,dx \cdot \sum_{i=1}^{L} \left( f_\text{exact}((x_j)) - f_\epsilon(x_j) \right)^2 }
}{
    \sqrt{\,dx \cdot \sum_{i=1}^n \left( f_\text{exact}((x_j)) \right)^2}
}.
\]
The error table below summarizes the absolute and relative errors between the exact and regularized solutions.
\begin{table}[H]
\centering
\begin{tabular}{|c|c|c|}
\hline
\textbf{epsilon} & \textbf{$E_{abs}$} & \textbf{$E_{rel}$} \\
\hline
$10^{-1}$ & 5.29$\times 10^{-1}$ & 47.26\% \\
$10^{-2}$ & 1.03$\times 10^{-1}$ & 9.22\%  \\
$10^{-3}$ & 5.07$\times 10^{-3}$ & 0.45\%  \\
$10^{-4}$ & 2.71$\times 10^{-3}$ & 0.24\%  \\
\hline
\end{tabular}
\end{table}

   	\section*{Acknowledgments} 
	We thank the reviewer for their suggestion to write a more detailed version, which led to a substantial improvement over the original manuscript. This research is support by Vietnam National University (VNU-HCM) under grant number T2024-18-01. The first author also completed the improved version of this paper during a sponsored visiting research period at Vietnam Institute for Advanced Studies in Mathematics (VIASM) from July 1 to August 30, 2025. 
	\bibliographystyle{refs}
	\bibliography{source_id_bib}

\begin{bibdiv}
\begin{biblist}

\bib{segatti2020fractional}{article}{
      author={{Antonio Segatti}},
      author={{Juan Luis V{\'a}zquez}},
       title={{On a Fractional Thin Film Equation}},
        date={2020},
     journal={{Advances in Nonlinear Analysis}},
      volume={9},
      number={1},
       pages={1516\ndash 1558},
}

\bib{doi:10.1137/0915067}{article}{
      author={Bailey, David~H.},
      author={Swarztrauber, Paul~N.},
       title={A fast method for the numerical evaluation of continuous fourier
  and laplace transforms},
        date={1994},
     journal={SIAM Journal on Scientific Computing},
      volume={15},
      number={5},
       pages={1105\ndash 1110},
      eprint={https://doi.org/10.1137/0915067},
         url={https://doi.org/10.1137/0915067},
}

\bib{bucur}{article}{
      author={{Claudia Bucur}},
      author={{Enrico Valdinoci}},
       title={{An Introduction to the Fractional Laplacian}},
        date={2016},
     journal={{Nonlocal Diffusion and Applications}},
       pages={7\ndash 37},
}

\bib{QN}{article}{
      author={{Danh Hua Quoc Nam}},
      author={{Le Dinh Long}},
      author={{Donal O'Regan}},
      author={{Tran Bao Ngoc}},
      author={{Nguyen Huy Tuan}},
       title={{Identification of the Right-Hand Side in a Bi-Parabolic Equation
  with Final Data}},
        date={2022},
     journal={{Applicable Analysis}},
      volume={101},
      number={4},
       pages={1157\ndash 1175},
}

\bib{AV}{book}{
      author={{Fuensanta Andreu-Vaillo}},
       title={{Nonlocal Diffusion Problems}},
   publisher={American Mathematical Soc.},
        date={2010},
      number={165},
}

\bib{grunau2020positivity}{article}{
      author={{Hans-Christoph Grunau}},
      author={{Nobuhito Miyake}},
      author={{Shinya Okabe}},
       title={{Positivity of Solutions to the Cauchy Problem for Linear and
  Semilinear Biharmonic Heat Equations}},
        date={2020},
     journal={{Advances in Nonlinear Analysis}},
      volume={10},
      number={1},
       pages={353\ndash 370},
}

\bib{cheng2020inverse}{article}{
      author={{Jin Cheng}},
      author={{Jijun Liu}},
       title={{An Inverse Source Problem for Parabolic Equations with Local
  Measurements}},
        date={2020},
     journal={{Applied Mathematics Letters}},
      volume={103},
       pages={106213},
}

\bib{greer2006fourth}{article}{
      author={{John B. Greer}},
      author={{Andrea L. Bertozzi}},
      author={{Guillermo Sapiro}},
       title={{Fourth Order Partial Differential Equations on General
  Geometries}},
        date={2006},
     journal={{Journal of Computational Physics}},
      volume={216},
      number={1},
       pages={216\ndash 246},
}

\bib{payne2006proposed}{article}{
      author={{L. E. Payne}},
      author={{J. C. Song}},
       title={{On a Proposed Model for Heat Conduction}},
        date={2006},
     journal={{IMA Journal of Applied Mathematics}},
      volume={71},
      number={4},
       pages={590\ndash 599},
}

\bib{wang2007heat}{book}{
      author={{Liqiu Wang}},
      author={{Xuesheng Zhou}},
      author={{Xiaohao Wei}},
       title={{Heat Conduction: Mathematical Models and Analytical Solutions}},
   publisher={Springer Science \& Business Media},
        date={2007},
}

\bib{nishio2021reproducing}{article}{
      author={{Masaharu Nishio}},
      author={{Katsunori Shimomura}},
       title={{Reproducing Property for Iterated Parabolic Operators of
  Fractional Order}},
        date={2021},
     journal={{Math. Reports, Romanian Academy}},
      volume={23},
      number={73},
       pages={1\ndash 2},
}

\bib{nishio2005alpha}{article}{
      author={{Masaharu Nishio}},
      author={{Katsunori Shimomura}},
      author={{Noriaki Suzuki}},
       title={{$\alpha$-Parabolic Bergman Spaces}},
        date={2005},
}

\bib{ebenbeck2020weak}{article}{
      author={{Matthias Ebenbeck}},
      author={{Kei Fong Lam}},
       title={{Weak and Stationary Solutions to a Cahn--Hilliard--Brinkman
  Model with Singular Potentials and Source Terms}},
        date={2020},
     journal={{Advances in Nonlinear Analysis}},
      volume={10},
      number={1},
       pages={24\ndash 65},
}

\bib{jorgensen2014}{misc}{
      author={{Michael Jorgensen}},
       title={{Volumes of $n$-Dimensional Spheres and Ellipsoids}},
        date={2014},
}

\bib{MN}{article}{
      author={{Mondal, Subhankar}},
      author={{Nair, M Thamban}},
       title={{A Source Identification Problem in a Bi-Parabolic Equation:
  Convergence Rates and Some Optimal Results}},
        date={2024},
     journal={{Numerical Functional Analysis and Optimization}},
      volume={45},
      number={3},
       pages={189\ndash 215},
}

\bib{DP}{article}{
      author={{Nguyen Duc Phuong}},
      author={{Nguyen Luc}},
      author={{others}},
       title={{Modified Quasi Boundary Value Method for Inverse Source
  Biparabolic}},
        date={2020},
     journal={{Advances in the Theory of Nonlinear Analysis and its
  Application}},
      volume={4},
      number={3},
       pages={132\ndash 142},
}

\bib{Tuan}{article}{
      author={{Nguyen Huy Tuan}},
       title={{On Some Inverse Problem for Bi-Parabolic Equation with Observed
  Data in $L^p$ Spaces}},
        date={2022},
     journal={{Opuscula Mathematica}},
      volume={42},
      number={2},
}

\bib{tuan2021initial}{article}{
      author={{Nguyen Huy Tuan}},
      author={{Tom{\'a}s Caraballo}},
       title={{On Initial and Terminal Value Problems for Fractional
  Nonclassical Diffusion Equations}},
        date={2021},
     journal={{Proceedings of the American Mathematical Society}},
      volume={149},
      number={1},
       pages={143\ndash 161},
}

\bib{ghoul2019construction}{article}{
      author={{Tej-Eddine Ghoul}},
      author={{Van Tien Nguyen}},
      author={{Hatem Zaag}},
       title={{Construction of Type I Blowup Solutions for a Higher Order
  Semilinear Parabolic Equation}},
        date={2019},
     journal={{Advances in Nonlinear Analysis}},
      volume={9},
      number={1},
       pages={388\ndash 412},
}

\bib{fushchich1990new}{article}{
      author={{V. I. Fushchich}},
      author={{A. S. Galitsyn}},
      author={{A. S. Polubinskii}},
       title={{A New Mathematical Model of Heat Conduction Processes}},
        date={1990},
     journal={{Ukrainian Mathematical Journal}},
      volume={42},
       pages={210\ndash 216},
}

\bib{bulavatsky2016fractional}{article}{
      author={{V. M. Bulavatsky}},
       title={{Fractional Differential Analog of Biparabolic Evolution Equation
  and Some Its Applications}},
        date={2016},
     journal={{Cybernetics and Systems Analysis}},
      volume={52},
       pages={737\ndash 747},
}

\bib{bulavatsky2008generalized}{article}{
      author={{V. M. Bulavatsky}},
      author={{V. V. Skopetskii}},
       title={{Generalized Mathematical Model of the Dynamics of Consolidation
  Processes with Relaxation}},
        date={2008},
     journal={{Cybernetics and Systems Analysis}},
      volume={44},
       pages={646\ndash 654},
}

\bib{Duc}{article}{
      author={{Van Duc, Nguyen}},
      author={{Van Thang, Nguyen}},
      author={{Nguyen Trung Th{\`a}nh}},
       title={{The Quasi-Reversibility Method for an Inverse Source Problem for
  Time-Space Fractional Parabolic Equations}},
        date={2023},
     journal={{Journal of Differential Equations}},
      volume={344},
       pages={102\ndash 130},
}

\bib{kalantarov2009finite}{article}{
      author={{Varga Kalantarov}},
      author={{Sergey Zelik}},
       title={{Finite-Dimensional Attractors for the Quasi-Linear
  Strongly-Damped Wave Equation}},
        date={2009},
     journal={{Journal of Differential Equations}},
      volume={247},
      number={4},
       pages={1120\ndash 1155},
}

\bib{pata2005strongly}{article}{
      author={{Vittorino Pata}},
      author={{Marco Squassina}},
       title={{On the Strongly Damped Wave Equation}},
        date={2005},
     journal={{Communications in Mathematical Physics}},
      volume={253},
      number={3},
       pages={511\ndash 533},
}

\bib{bulavatskiy2003mathematical}{article}{
      author={{Vladimir M. Bulavatskiy}},
       title={{Mathematical Modeling of Filtrational Consolidation of Soils
  under Motion of Saline Solutions on the Basis of Biparabolic Model}},
        date={2003},
     journal={{Journal of Automation and Information Sciences}},
      volume={35},
      number={8},
}

\bib{yanbing2019global}{article}{
      author={{Yang Yanbing}},
      author={{Md Salik Ahmed}},
      author={{Qin Lanlan}},
      author={{Xu Runzhang}},
       title={{Global Well-Posedness of a Class of Fourth-Order Strongly Damped
  Nonlinear Wave Equations}},
        date={2019},
     journal={{Opuscula Mathematica}},
      volume={39},
      number={2},
}

\bib{hishikawa2023remark}{article}{
      author={{Y{\^o}suke Hishikawa}},
      author={{Masaharu Nishio}},
      author={{Katsunori Shimomura}},
      author={{Masahiro Yamada}},
       title={{A Remark on the Dual Spaces of Bi-Parabolic Bergman Spaces
  Dedicated to Professor Eiichi Nakai on the Occasion of His 65th Birthday}},
        date={2023},
     journal={{Mathematical Journal of Ibaraki University}},
      volume={55},
       pages={1\ndash 16},
}

\end{biblist}
\end{bibdiv}
\end{document}